\documentclass[11pt]{article}
\pdfoutput=1

\usepackage[margin=1in]{geometry}

\usepackage{amsfonts} 
\usepackage{epsfig}
\usepackage{hyphenat}
\newcommand{\norm}[1]{\left\lVert#1\right\rVert}

\usepackage{mathrsfs}
\usepackage{algorithm}
\usepackage{algorithmic}
\usepackage{tocvsec2}
\usepackage{caption}
\usepackage{makecell}
\usepackage{epsfig}%
\usepackage[namelimits]{amsmath}%

\usepackage{subfigmat}  
\usepackage{cases}%
\usepackage{amssymb}%
\usepackage{bm}%
\usepackage{xspace}%
\usepackage{array}%
\usepackage{rotating}
\usepackage[table]{xcolor}
\usepackage{tabularx,colortbl}
\usepackage{float}
\usepackage{textcomp}%
\usepackage{caption}

\usepackage{booktabs, multicol, multirow}

\usepackage{amsthm}
\usepackage{cite}
\usepackage{hyperref}

\usepackage[nameinlink]{cleveref}

\usepackage{amsmath}
\usepackage{dsfont}     
\graphicspath{{figures/}}

\usepackage{amsthm}

\theoremstyle{definition}
\newtheorem{theorem}{Theorem}

\newtheorem{proposition}{Proposition}
\newtheorem{lemma}{Lemma}

\newtheorem{example}{Example}

\begin{document}

\title{An HDG Method for Distributed Control of Convection Diffusion PDEs}

\author{Weiwei Hu%
	\thanks{Department of Mathematics, Oklahoma State
		University, Stillwater, OK (weiwei.hu@okstate.edu). W.~Hu was supported in part by a postdoctoral fellowship for the annual program on Control Theory and its Applications at the Institute for Mathematics and its Applications (IMA) at the University of Minnesota.}%
	\and
	Jiguang Shen%
	\thanks{School  of Mathematics, University of Minnesota, MN (shenx179@umn.edu)}%
	\and
	John~R.~Singler%
	\thanks{Department of Mathematics
		and Statistics, Missouri University of Science and Technology,
		Rolla, MO (\mbox{singlerj@mst.edu}, ywzfg4@mst.edu). J.~Singler and Y.~Zhang were supported in part by National Science Foundation grant DMS-1217122.  J.~Singler and Y.~Zhang thank the IMA for funding research visits, during which some of this work was completed.}
	\and
	Yangwen Zhang%
	\footnotemark[3]
	\and
	Xiaobo Zheng%
	\thanks{College of Mathematics, Sichuan University,
		Chengdu, China (zhengxiaobosc@yahoo.com). X.~Zheng thanks Missouri University of Science and Technology for hosting him as a visiting scholar; some of this work was completed during his research visit.}
}

\date{}

\maketitle

\begin{abstract}
We propose a  hybridizable discontinuous Galerkin (HDG) method to approximate the solution of a distributed optimal control problem governed by an elliptic convection diffusion PDE. We derive optimal a priori error estimates for the state, adjoint state, their fluxes, and the optimal control. We present 2D and 3D numerical experiments to illustrate our theoretical results. 
\end{abstract}


\section{Introduction}
\label{intro}
We consider the following  distributed control problem:  Minimize the functional 
\begin{align}
\min J(u)=\frac{1}{2}\| y- y_{d}\|^2_{L^{2}(\Omega)}+\frac{\gamma}{2}\|u\|^2_{L^{2}(\Omega)}, \quad \gamma>0, \label{cost1}
\end{align}
subject to
\begin{equation}\label{Ori_problem}
\begin{split}
-\Delta y+\bm \beta\cdot\nabla y&=f+u \quad\text{in}~\Omega,\\
y&=g\qquad\quad\text{on}~\partial\Omega,
\end{split}
\end{equation}
where $\Omega\subset \mathbb{R}^{d} $ $ (d\geq 2)$ is a Lipschitz polyhedral domain  with boundary $\Gamma = \partial \Omega$, $ f \in L^2(\Omega) $, and the vector field $\bm{\beta}$ satisfies
\begin{align}\label{beta_con}
\nabla\cdot\bm{\beta}  = 0.
\end{align}
It is well known that the optimal control problem \eqref{cost1}-\eqref{Ori_problem} is equivalent to the optimality system
\begin{subequations}\label{eq_adeq}
	\begin{align}
	-\Delta y+\bm \beta\cdot\nabla y &=f+u\quad~\text{in}~\Omega,\label{eq_adeq_a}\\
	y&=g\qquad~~~~\text{on}~\partial\Omega,\label{eq_adeq_b}\\
	-\Delta z-\nabla\cdot(\bm{\beta} z) &=y_d-y\quad~\text{in}~\Omega,\label{eq_adeq_c}\\
	z&=0\qquad\quad~~\text{on}~\partial\Omega,\label{eq_adeq_d}\\
	z-\gamma  u&=0\qquad\quad~~\text{in}~\Omega.\label{eq_adeq_e}
	\end{align}
\end{subequations}

Optimal control problems for convection diffusion equations arise in applications \cite{MR1766429} and are also an important step towards optimal control problems for fluid flows.  Therefore, researchers have developed many different numerical methods for this type of problem including approaches based on finite differences \cite{MR3451511}, standard finite element discretizations \cite{MR2851444,MR2719819,MR2475653}, stabilized finite elements \cite{MR2302057,MR3144702}, the symmetric stabilization method \cite{MR2486088}, the SUPG method \cite{MR2595051,MR3451479}, the edge-stabilization method \cite{MR2463111,MR2068903}, mixed finite elements \cite{MR2550371,MR2851444,MR2971662}, and discontinuous Galerkin (DG) methods \cite{MR2595051,MR2587414,MR2773301,MR3416418,MR3149415,MR3022208,MR2644299}.

DG methods are well suited for problems with convection, but they often have a higher computational cost compared to other methods.  Hybridizable discontinuous Galerkin (HDG) methods keep the advantages of DG methods, but have a lower number of globally coupled unknowns.  HDG methods were introduced in \cite{MR2485455}, and now have been applied to many different problems \cite{MR2772094,MR2513831,MR2558780,MR2796169,MR3626531,MR3522968,MR3463051,MR3452794,MR3343926}.

HDG methods have recently been successfully applied to two PDE optimal control problems.  Zhu and Celiker \cite{MR3508834} obtained optimal convergence rates for an HDG method for a distributed optimal control problem governed by the Poisson equation.  The authors have also studied an HDG method for a difficult Dirichlet optimal boundary control problem for the Poisson equation in \cite{HuShenSinglerZhangZheng_HDG_Dirichlet_control1}.  We proved an optimal superlinear convergence rate for the control in polygonal domains.  Despite the large amount of work on this problem, a superlinear convergence result of this type had only been previously obtained for one other numerical method on a special class of meshes \cite{ApelMateosPfeffererRosch17}.

Due to these recent results and the favorable properties of HDG methods, we continue to investigate HDG for optimal control problems for PDEs in this work.  Specifically, we consider the above distributed control problem for the elliptic convection diffusion equation, and apply an HDG method with polynomials of degree $k$ to approximate all the variables of the optimality system \eqref{eq_adeq}, i.e., the state $y$, dual state $z$, the numerical traces, and the fluxes $\bm q = -\nabla  y $ and $ \bm p = -\nabla z$.  We describe the HDG method and its implementation in \Cref{sec:HDG}.  In \Cref{sec:analysis}, we obtain the error estimates
\begin{align*}
&\norm{y-{y}_h}_{0,\Omega}=O( h^{k+1}),\quad \quad   \;\norm{z-{z}_h}_{0,\Omega}=O( h^{k+1}),\\
&\norm{\bm{q}-\bm{q}_h}_{0,\Omega} = O( h^{k+1}),\quad  \quad\;\; \norm{\bm{p}-\bm{p}_h}_{0,\Omega} = O( h^{k+1}),
\end{align*}
and
\begin{align*}
&\norm{u-{u}_h}_{0,\Omega} = O( h^{k+1}).
\end{align*}
We present 2D and 3D numerical results in \Cref{sec:numerics} and then briefly discuss future work.

\section{HDG scheme for the optimal control problem}
\label{sec:HDG}

We begin by setting notation.

Throughout the paper we adopt the standard notation $W^{m,p}(\Omega)$ for Sobolev spaces on $\Omega$ with norm $\|\cdot\|_{m,p,\Omega}$ and seminorm $|\cdot|_{m,p,\Omega}$ . We denote $W^{m,2}(\Omega)$ by $H^{m}(\Omega)$ with norm $\|\cdot\|_{m,\Omega}$ and seminorm $|\cdot|_{m,\Omega}$. Specifically, $H_0^1(\Omega)=\{v\in H^1(\Omega):v=0 \;\mbox{on}\; \partial \Omega\}$.  We denote the $L^2$-inner products on $L^2(\Omega)$ and $L^2(\Gamma)$ by
\begin{align*}
(v,w) &= \int_{\Omega} vw  \quad \forall v,w\in L^2(\Omega),\\
\left\langle v,w\right\rangle &= \int_{\Gamma} vw  \quad\forall v,w\in L^2(\Gamma).
\end{align*}
Define the space $H(\text{div},\Omega)$ as
\begin{align*}
H(\text{div},\Omega) = \{\bm{v}\in [L^2(\Omega)]^d, \nabla\cdot \bm{v}\in L^2(\Omega)\}.
\end{align*}

Let $\mathcal{T}_h$ be a collection of disjoint elements that partition $\Omega$.  We denote by $\partial \mathcal{T}_h$ the set $\{\partial K: K\in \mathcal{T}_h\}$. For an element $K$ of the collection  $\mathcal{T}_h$, let $e = \partial K \cap \Gamma$ denote the boundary face of $ K $ if the $d-1$ Lebesgue measure of $e$ is non-zero. For two elements $K^+$ and $K^-$ of the collection $\mathcal{T}_h$, let $e = \partial K^+ \cap \partial K^-$ denote the interior face between $K^+$ and $K^-$ if the $d-1$ Lebesgue measure of $e$ is non-zero. Let $\varepsilon_h^o$ and $\varepsilon_h^{\partial}$ denote the set of interior and boundary faces, respectively. We denote by $\varepsilon_h$ the union of  $\varepsilon_h^o$ and $\varepsilon_h^{\partial}$. We finally introduce
\begin{align*}
(w,v)_{\mathcal{T}_h} = \sum_{K\in\mathcal{T}_h} (w,v)_K,   \quad\quad\quad\quad\left\langle \zeta,\rho\right\rangle_{\partial\mathcal{T}_h} = \sum_{K\in\mathcal{T}_h} \left\langle \zeta,\rho\right\rangle_{\partial K}.
\end{align*}

Let $\mathcal{P}^k(D)$ denote the set of polynomials of degree at most $k$ on a domain $D$.  We introduce the discontinuous finite element spaces
\begin{align}
\bm{V}_h  &:= \{\bm{v}\in [L^2(\Omega)]^d: \bm{v}|_{K}\in [\mathcal{P}^k(K)]^d, \forall K\in \mathcal{T}_h\},\\
{W}_h  &:= \{{w}\in L^2(\Omega): {w}|_{K}\in \mathcal{P}^{k}(K), \forall K\in \mathcal{T}_h\},\\
{M}_h  &:= \{{\mu}\in L^2(\mathcal{\varepsilon}_h): {\mu}|_{e}\in \mathcal{P}^k(e), \forall e\in \varepsilon_h\}.
\end{align}
Let  $M_h(o)$ and $M_h(\partial)$  denote the subspaces of $M_h$ containing each $e\in \varepsilon_h^o$ and  $e\in \varepsilon_h^{\partial}$, respectively. Note that $M_h$ consists of functions which are continuous inside the faces (or edges) $e\in \varepsilon_h$ and discontinuous at their borders. In addition, for any function $w\in W_h$ we use $\nabla w$ to denote the piecewise gradient on each element $K\in \mathcal T_h$. A similar convention applies to the divergence $\nabla\cdot\bm r$ for all $\bm r\in \bm V_h$.

\subsection{The HDG Formulation}

The mixed weak form of the optimality system \eqref{eq_adeq_a}-\eqref{eq_adeq_e} is given by
\begin{subequations}\label{mixed}
	\begin{align}
	(\bm q,\bm r_1)-( y,\nabla\cdot \bm r_1)+\langle y,\bm r_1\cdot \bm n\rangle&=0,\label{mixed_a}\\
	(\nabla\cdot(\bm q+\bm \beta y),  w_1)&= ( f+ u, w_1),  \label{mixed_b}\\
	(\bm p,\bm r_2)-(z,\nabla \cdot\bm r_2)+\langle z,\bm r_2\cdot\bm n\rangle&=0,\label{mixed_c}\\
	(\nabla\cdot(\bm p-\bm \beta z),  w_2)&= (y_d- y, w_2),  \label{mixed_d}\\
	( z-\gamma u,v)&=0,\label{mixed_e}
	\end{align}
\end{subequations}
for all $(\bm r_1, w_1,\bm r_2, w_2,v)\in H(\text{div},\Omega)\times L^2(\Omega)\times H(\text{div},\Omega)\times L^2(\Omega)\times L^2(\Omega)$.  Recall we assume $ \bm \beta $ is divergence free; this allows us to rewrite the convection term $ \bm \beta \cdot \nabla y $ in \eqref{eq_adeq_a} as $ \nabla \cdot( \bm \beta y ) $ in \eqref{mixed_b}.

To approximate the solution of this system, the HDG method seeks approximate fluxes ${\bm{q}}_h,{\bm{p}}_h \in \bm{V}_h $, states $ y_h, z_h \in W_h $, interior element boundary traces $ \widehat{y}_h^o,\widehat{z}_h^o \in M_h(o) $, and  control $ u_h \in W_h$ satisfying
\begin{subequations}\label{HDG_discrete2}
	\begin{align}
	(\bm q_h,\bm r_1)_{\mathcal T_h}-( y_h,\nabla\cdot\bm r_1)_{\mathcal T_h}+\langle \widehat y_h^o,\bm r_1\cdot\bm n\rangle_{\partial\mathcal T_h\backslash \varepsilon_h^\partial}&=-\langle  g,\bm r_1\cdot\bm n\rangle_{\varepsilon_h^\partial}, \label{HDG_discrete2_a}\\
	-(\bm q_h+\bm \beta y_h,  \nabla w_1)_{\mathcal T_h} +\langle\widehat {\bm q}_h\cdot\bm n,w_1\rangle_{\partial\mathcal T_h}  \quad  \nonumber \\ 
	 +\langle \bm \beta\cdot\bm n\widehat y_h^o,w_1\rangle_{\partial\mathcal T_h\backslash\varepsilon_h^\partial} -  ( u_h, w_1)_{\mathcal T_h} &=  - \langle \bm \beta\cdot\bm n g,w_1\rangle_{\varepsilon_h^\partial} + ( f, w_1)_{\mathcal T_h}   \label{HDG_discrete2_b}
	%
	\end{align}
	for all $(\bm{r}_1, w_1)\in \bm{V}_h\times W_h$.
	\begin{align}
	(\bm p_h,\bm r_2)_{\mathcal T_h}-(z_h,\nabla\cdot\bm r_2)_{\mathcal T_h}+\langle \widehat z_h^o,\bm r_2\cdot\bm n\rangle_{\partial\mathcal T_h\backslash\varepsilon_h^\partial}&=0,\label{HDG_discrete2_c}\\
	-(\bm p_h-\bm \beta z_h, \nabla w_2)_{\mathcal T_h}+\langle\widehat{\bm p}_h\cdot\bm n,w_2\rangle_{\partial\mathcal T_h}  \quad  \nonumber\\
	-\langle\bm \beta\cdot\bm n\widehat z_h^o,w_2\rangle_{\partial\mathcal T_h\backslash \varepsilon_h^\partial} +  ( y_h, w_2)_{\mathcal T_h}&= (y_d, w_2)_{\mathcal T_h},  \label{HDG_discrete2_d}
	\end{align}
	for all $(\bm{r}_2, w_2)\in \bm{V}_h\times W_h$.
	\begin{align}
	\langle\widehat {\bm q}_h\cdot\bm n+\bm \beta\cdot\bm n\widehat y_h^o,\mu_1\rangle_{\partial\mathcal T_h\backslash\varepsilon^{\partial}_h}&=0\label{HDG_discrete2_e},\\
	\langle\widehat{\bm p}_h\cdot\bm n-\bm \beta\cdot\bm n\widehat z_h^o,\mu_2\rangle_{\partial\mathcal T_h\backslash\varepsilon^{\partial}_h}&=0,\label{HDG_discrete2_f}
	\end{align}
	for all $\mu_1,\mu_2\in M_h(o)$, and the optimality condition 
	\begin{align}
	(z_h-\gamma u_h, w_3)_{\mathcal T_h} &= 0\label{HDG_discrete2_g},
	\end{align}
	for all $ w_3\in W_h$.  The numerical traces on $\partial\mathcal{T}_h$ are defined as 
	\begin{align}
	\widehat{\bm{q}}_h\cdot \bm n &=\bm q_h\cdot\bm n+\tau_1 (y_h-\widehat y_h^o)   \quad \mbox{on} \; \partial \mathcal{T}_h\backslash\varepsilon_h^\partial, \label{HDG_discrete2_h}\\
	\widehat{\bm{q}}_h\cdot \bm n &=\bm q_h\cdot\bm n+\tau_1 (y_h-g)  \quad \   \mbox{on}\;  \varepsilon_h^\partial, \label{HDG_discrete2_i}\\
	\widehat{\bm{p}}_h\cdot \bm n &=\bm p_h\cdot\bm n+\tau_2(z_h-\widehat z_h^o)\quad \mbox{on} \; \partial \mathcal{T}_h\backslash\varepsilon_h^\partial,\label{HDG_discrete2_j}\\
	\widehat{\bm{p}}_h\cdot \bm n &=\bm p_h\cdot\bm n+\tau_2 z_h\quad\quad\quad\quad\mbox{on}\;  \varepsilon_h^\partial,\label{HDG_discrete2_k}
	\end{align}
\end{subequations}
where $\tau_1$ and $\tau_2$ are positive stabilization functions defined on $\partial \mathcal T_h$.  We specify these functions in the next section.

\subsection{Implementation}
For the numerical implementation, we follow a similar procedure to our earlier work \cite{HuShenSinglerZhangZheng_HDG_Dirichlet_control1}.  First, we perform some basic manipulations to the above system \eqref{HDG_discrete2_a}-\eqref{HDG_discrete2_k} to find that
\[ ({\bm{q}}_h,{\bm{p}}_h, y_h,z_h,{\widehat{y}}_h^o,{\widehat{z}}_h^o)\in \bm{V}_h\times\bm{V}_h\times W_h \times W_h\times M_h(o)\times M_h(o) \]
is the solution of the following weak formulation:
\begin{subequations}\label{imple}
	\begin{align}
	(\bm{q}_h, \bm{r_1})_{{\mathcal{T}_h}}- (y_h, \nabla\cdot \bm{r_1})_{{\mathcal{T}_h}}+\langle \widehat{y}_h^o, \bm{r_1}\cdot \bm{n} \rangle_{\partial{{\mathcal{T}_h}}\backslash \varepsilon_h^{\partial}} &= - \langle g, \bm{r_1}\cdot \bm{n} \rangle_{\varepsilon_h^{\partial}}  , \label{imple_a}\\
	(\bm{p}_h, \bm{r_2})_{{\mathcal{T}_h}}- (z_h, \nabla\cdot \bm{r_2})_{{\mathcal{T}_h}}+\langle \widehat{z}_h^o, \bm{r_2}\cdot \bm{n} \rangle_{\partial{{\mathcal{T}_h}}\backslash \varepsilon_h^{\partial}} &=0, \label{imple_b}\\
	(\nabla\cdot\bm{q}_h,  w_1)_{{\mathcal{T}_h}} - (\bm{\beta} y_h, \nabla w_1)_{\mathcal T_h}   + \langle \tau_1 y_h, w_1\rangle_{\partial\mathcal T_h} \quad  \nonumber\\
	- (\gamma^{-1} z_h,w_1)_{\mathcal T_h}  +\langle  (\bm{\beta}\cdot\bm n - \tau_1)\widehat y_h^o, w_1 \rangle_{\partial{{\mathcal{T}_h}}\backslash\varepsilon_h^\partial} &= -\langle  (\bm{\beta}\cdot\bm n - \tau_1) g, w_1 \rangle_{\varepsilon_h^\partial}  \nonumber\\ 
	&\quad + (f, w_1)_{{\mathcal{T}_h}}, \label{imple_c}\\
	(\nabla\cdot\bm{p}_h,  w_2)_{{\mathcal{T}_h}} + (y_h, w_2)_{\mathcal T_h} + (\bm{\beta} z_h, \nabla w_2)_{\mathcal T_h} \quad &  \nonumber\\
	+ \langle \tau_2 z_h, w_2\rangle_{\partial \mathcal T_h} -\langle (\tau_2 + \bm{\beta}\cdot \bm n) \widehat{z}_h^o, w_2 \rangle_{\partial{{\mathcal{T}_h}}\backslash \varepsilon_h^{\partial}} &=(y_d, w_2)_{{\mathcal{T}_h}}, \label{imple_d}\\
	\langle{\bm{q}_h}\cdot \bm{n}, \mu_1 \rangle_{\partial\mathcal{T}_{h}\backslash {\varepsilon_h^{\partial}}}+\langle \tau_1 y_h ,\mu_1 \rangle_{\partial\mathcal{T}_{h}\backslash {\varepsilon_h^{\partial}}} \quad \nonumber \\
	+\langle (\bm{\beta}\cdot\bm n - \tau_1)\widehat{y}_h^{o} ,\mu_1 \rangle_{\partial\mathcal{T}_{h}\backslash {\varepsilon_h^{\partial}}}&=0, \label{imple_e}\\
	\langle {\bm{p}_h}\cdot \bm{n}, \mu_2\rangle_{ \partial\mathcal{T}_{h}\backslash {\varepsilon_h^{\partial}}}+\langle \tau_2 z_h, \mu_2\rangle_{ \partial\mathcal{T}_{h}\backslash {\varepsilon_h^{\partial}}} \quad \nonumber\\
	-\langle (\bm{\beta}\cdot\bm n+\tau_2) \widehat{z}_h^o, \mu_2\rangle_{ \partial\mathcal{T}_{h}\backslash {\varepsilon_h^{\partial}}} &=0,
	\label{imple_f}
	\end{align}
\end{subequations}
for all  $({\bm{r}_1},{\bm{r}_2},w_1,w_2,\mu_1,\mu_2)\in \bm{V}_h\times\bm{V}_h\times W_h \times W_h\times M_h(o)\times M_h(o)$.

Note that we have used the optimality condition \eqref{HDG_discrete2_g} to eliminate $ u_h $ from the discrete equations.  Once the above system \eqref{imple} is solved numerically, $ u_h $ can be easily found using the optimality condition: $ u_h = \gamma^{-1} z_h $.

\subsection{Matrix equations}
Assume $\bm{V}_h = \mbox{span}\{\bm\varphi_i\}_{i=1}^{N_1}$, $W_h=\mbox{span}\{\phi_i\}_{i=1}^{N_2}$, $M_h^{o}=\mbox{span}\{\psi_i\}_{i=1}^{N_3} $. Then
\begin{equation}\label{expre}
\begin{split}
&\bm q_{h}= \sum_{j=1}^{N_1}q_{j}\bm\varphi_j,  \quad y_h = \sum_{j=1}^{N_2}y_{j}\phi_j, \quad \widehat{y}_h^o = \sum_{j=1}^{N_3}\alpha_{j}\psi_{j},\\
&\bm p_{h} =  \sum_{j=1}^{N_1}p_{j}\bm\varphi_j, \quad z_h =  \sum_{j=1}^{N_2} z_{j}\phi_j, \quad  \widehat{z}_h^o = \sum_{j=1}^{N_3}\gamma_{j}\psi_{j}.
\end{split}
\end{equation}
Substitute \eqref{expre} into \eqref{imple_a}-\eqref{imple_f} and use the corresponding  test functions to test \eqref{imple_a}-\eqref{imple_f}, respectively, to obtain the matrix equation
\begin{align}\label{system_equation}
\begin{bmatrix}
A_1  &0 &-A_2&0  & A_{15}&0 \\
0 & A_1 &0 &-A_2& 0  &A_{15} \\
A_2^T &0&A_{12}&-\gamma^{-1}A_4&A_{16}  &0\\
0 &A_2^T & A_4 &A_{13} &0&A_{17} \\
A_{14}^T & 0 &A_{18}&0 &A_{20}&0  \\
0& A_{15}^T &0&A_{19} &0&A_{21}  
\end{bmatrix}
\left[ {\begin{array}{*{20}{c}}
	\mathfrak{q}\\
	\mathfrak{p}\\
	\mathfrak{y}\\
	\mathfrak{z}\\
	\mathfrak{\widehat y}\\
	\mathfrak{\widehat z}
	\end{array}} \right]
=\left[ {\begin{array}{*{20}{c}}
	-b_1\\
	0\\
	-b_5\\
	b_4\\
	0\\
	0\\
	\end{array}} \right],
\end{align}
where $\mathfrak{q},\mathfrak{p},\mathfrak{y},\mathfrak{z},\mathfrak{\widehat y},\mathfrak{\widehat z}$ are the coefficient vectors for $\bm q_h,\bm p_h,y_h,z_h,\widehat y_h^o, \widehat z_h^o$, respectively,  and
\begin{align*}
A_1 &= [(\bm\varphi_j,\bm\varphi_i )_{\mathcal{T}_h}],   &
A_2 &= [(\phi_j,\nabla\cdot\bm{\varphi_i})_{\mathcal{T}_h}],   &
A_3 &= [(\psi_j,\bm{\varphi}_i\cdot \bm n)_{\mathcal{T}_h}],\\
A_4 &= [(\phi_j,\phi_i)_{\mathcal{T}_h}],  &
A_5 &= [(\bm\beta\phi_j,\nabla\phi_i)_{\mathcal{T}_h}], &
A_6 &= [\langle  \tau_1\phi_j, \phi_i \rangle_{\partial{{\mathcal{T}_h}}}],\\
A_7 &= [\langle  \bm{\beta}\cdot\bm n\phi_j, \phi_i \rangle_{\partial{{\mathcal{T}_h}}}],  &
A_8 &=  [\left\langle \tau_1\psi_j,{\varphi_i}\right\rangle_{\partial\mathcal{T}_h}],  &
A_9 &=  [\left\langle \bm{\beta}\cdot\bm n\psi_j,{\varphi_i}\right\rangle_{\partial\mathcal{T}_h}],\\ 
A_{10} &= [\left\langle \tau_1\psi_j,\psi_i\right\rangle_{\partial\mathcal{T}_h}], & 
A_{11} &= [\left\langle \bm{\beta}\cdot\bm n \psi_j,\psi_i\right\rangle_{\partial\mathcal{T}_h}], &
A_{12} &= A_6-A_5\\
A_{13} &= A_5+ A_6-A_7, &
b_1 &= [\langle g, \bm\varphi_i \cdot \bm{n} \rangle_{\varepsilon_h^{\partial}}],&
b_2 &= [\langle  (\bm{\beta}\cdot\bm n - \tau_1) g, \phi \rangle_{\varepsilon_h^\partial}], \\
b_3 &= [(f,\phi_i )_{\mathcal{T}_h}], &
b_4 &= [(y_d,\phi_i )_{\mathcal{T}_h}], &
b_5 &= b_3-b_2.
\end{align*}
The remaining matrices are constructed by  extracting the corresponding rows and columns from linear combinations of $A_3 $, $A_8$, $A_{9}$, $A_{10}$, and $A_{11}$.

\subsection{Local solver}

Next, we use the discontinuous nature of the approximation spaces $\bm{V_h}$ and ${W_h}$ to eliminate all unknowns except the coefficient vectors of the numerical traces.  

The matrix equation \eqref{system_equation} can be rewritten as
\begin{align}\label{system_equation2}
\begin{bmatrix}
B_1 & B_2&B_3\\
-B_2^T & B_4&B_5\\
B_6&B_7&B_8\\
\end{bmatrix}
\left[ {\begin{array}{*{20}{c}}
	\bm{\alpha}\\
	\bm{\beta}\\
	\bm{\gamma}
	\end{array}} \right]
=\left[ {\begin{array}{*{20}{c}}
	\bm b_1\\
	\bm b_2\\
	0
	\end{array}} \right],
\end{align}
where $\bm{\alpha}=[\mathfrak{q};\mathfrak{p}]$, $\bm{\beta}=[\mathfrak{y};\mathfrak{z}]$, $\bm{\gamma}=[\mathfrak{\widehat y};\mathfrak{\widehat z}]$, $ \bm b_1 = [- b_1; 0 ] $, and $ \bm b_2 = [ -b_5; b_4 ] $, and also $\{B_i\}_{i=1}^8$ are the corresponding blocks of the coefficient matrix of \eqref{system_equation}.

In the appendix, we show how the first two equations of  \eqref{system_equation2} can be used to eliminate both $\bm{\alpha}$ and $\bm{\beta}$ in an element-by-element fashion.  We obtain
\begin{align}\label{local_solver}
\left[ {\begin{array}{*{20}{c}}
	\bm{\alpha}\\
	\bm{\beta}
	\end{array}} \right]= \begin{bmatrix}
G_1 & H_1\\
G_2 & H_2
\end{bmatrix}
\left[ {\begin{array}{*{20}{c}}
	\bm \gamma\\
	b
	\end{array}} \right]
\end{align}
and
\begin{align}
B_6\bm{\alpha}+B_7\bm{\beta}+B_8\bm{\gamma} = 0,
\end{align}
where $G_1,G_2,H_1,H_2$ are sparse.  This gives a globally coupled equation for $\bm{\gamma}$ only:
\begin{align}\label{global_eq}
\mathbb{K}\bm{\gamma} = \mathbb{F},
\end{align}
where
\begin{align*}
\mathbb{K}= B_6G_1+B_7G_2+B_8\quad\text{and}\quad\mathbb{F} = B_6H_1+B_7H_2.
\end{align*}
Once $\bm{\gamma}$ is computed, $\bm{\alpha}$ and $\bm{\beta}$ can be quickly and easily computed using \eqref{local_solver}.

\section{Error Analysis}
\label{sec:analysis}

Next, we provide a convergence analysis of the above HDG method for the optimal control problem.  Throughout this section, we assume $ \bm \beta \in [W^{1,\infty}(\Omega)]^d $, $ \Omega $ is a bounded convex polyhedral domain, $ h \leq 1 $, and the solution of the optimality system \eqref{eq_adeq} is smooth enough.  

\subsection{Main result}
For our theoretical results, we require the stabilization functions $\tau_1$ and $\tau_2$ are chosen to satisfy
\begin{description}
	
	\item[\textbf{(A1)}] $\tau_1$ is piecewise constant on $\partial \mathcal T_h$.
	
	\item[\textbf{(A2)}] $\tau_1 = \tau_2 + \bm{\beta}\cdot \bm n$.
	
	\item[\textbf{(A3)}] For any  $K\in\mathcal T_h$, $\min{(\tau_1-\frac 1 2 \bm \beta \cdot \bm n)}|_{\partial K} >0$.
	
\end{description}
We note that \textbf{(A2)} and \textbf{(A3)} imply
\begin{equation}\label{eqn:tau1_condition}
\min{(\tau_2 + \frac 1 2 \bm \beta \cdot \bm n)}|_{\partial K} >0  \quad  \mbox{for any $K\in\mathcal T_h$.}
\end{equation}

\begin{theorem}\label{main_res}
	We have
	\begin{align*}
	\|\bm q-\bm q_h\|_{\mathcal T_h}&\lesssim h^{k+1}(|\bm q|_{k+1}+|y|_{k+1}+|\bm p|_{k+1}+|z|_{k+1}),\\
	\|\bm p-\bm p_h\|_{\mathcal T_h}&\lesssim h^{k+1}(|\bm q|_{k+1}+|y|_{k+1}+|\bm p|_{k+1}+|z|_{k+1}),\\
	\|y-y_h\|_{\mathcal T_h}&\lesssim h^{k+1}(|\bm q|_{k+1}+|y|_{k+1}+|\bm p|_{k+1}+|z|_{k+1}),\\
	\|z-z_h\|_{\mathcal T_h}&\lesssim h^{k+1}(|\bm q|_{k+1}+|y|_{k+1}+|\bm p|_{k+1}+|z|_{k+1}),\\
	\|u-u_h\|_{\mathcal T_h}&\lesssim h^{k+1}(|\bm q|_{k+1}+|y|_{k+1}+|\bm p|_{k+1}+|z|_{k+1}).
	\end{align*}
\end{theorem}

\subsection{Preliminary material}
\label{sec:Projectionoperator}
Next, we  introduce the projection operators $\bm{\Pi}_V$ and $\Pi_W$ defined in \cite{MR2991828} that we use frequently in our proof.  The value of the projection on each element $K\in \mathcal{T}_h$ is determined by requiring that the components satisfy the equations
\begin{subequations}\label{projection_operator}
	\begin{align}
	(\bm\Pi_V\bm q+\bm\beta\Pi_Wy,\bm r)_K&=(\bm q+\bm\beta y,\bm r)_K, \label{projection_operator_a}\\
	(\Pi_Wy,w)_K&=(y,w)_K, \label{projection_operator_b}\\
	\langle\bm\Pi_V\bm q\cdot\bm n+\bm\beta\cdot\bm nP_My+\tau_1\Pi_Wy,\mu\rangle_{e}&=\langle\bm q\cdot\bm n+\bm\beta\cdot\bm ny+\tau_1 y,\mu\rangle_{e}, \label{projection_operator_c}
	\end{align}
	for all  $(\bm{r} , w, \mu)\in \bm{\mathcal P}_{k-1}(K)\times {\mathcal P}_{k-1}(K)\times  {\mathcal P}_{k}(e)$ and for all faces $e$ of the simplex $K$.
\end{subequations}
Here, $P_M$ denotes the $L^2$-orthogonal projection  from $L^2(\varepsilon_h)$ into $M_h$ satisfying
\begin{align}\label{P_M}
\left\langle P_M y-y, \mu\right\rangle_e = 0, \quad \forall  e\in \varepsilon_h, \;\forall\mu\in M_h.
\end{align}
The following lemma from \cite{MR2991828} provides the approximation properties of the projection operator \eqref{projection_operator}.
\begin{lemma}\label{pro_error}
	Suppose $k\ge0$, and $ \tau_1 $ satisfies \textbf{(A3)}. Then the system \eqref{projection_operator} is uniquely solvable for $\bm{\Pi}_V\bm{q}$ and $\Pi_W y$. Moreover, we have the following approximation properties
	\begin{subequations}
		\begin{align}
		\|\bm\Pi_V\bm q-\bm q\|_K &\le Ch^{k+1}|\bm q|_{k+1,K}+Ch^{k+1}|y|_{k+1,K},\\
		\|\Pi_Wy-y\|_K&\le Ch^{k+1}|\bm q|_{k+1,K}+Ch^{k+1}|y|_{k+1,K},
		\end{align}
	\end{subequations}
	where $C$ is a constant depending on the polynomial degree and the shape-regularity parameters of the elements.
\end{lemma}
For the convection diffusion optimal control problem, we introduce another projection operator associated to the dual problem.  The projection $\widetilde{\bm \Pi}_{V}$ and $\widetilde{\Pi}_W$ is determined by the following equations 
\begin{subequations}\label{projection_operator1}
	\begin{align}
	(\widetilde{\bm\Pi}_V\bm p-\bm\beta\widetilde{\Pi}_Wz,\bm r)_K&=(\bm p-\bm\beta z,\bm r)_K, \label{projection_operator1_a}\\
	(\widetilde {\Pi}_Wz,w)_K&=(z,w)_K,\label{projection_operator1_b}\\
	\langle\widetilde{\bm\Pi}_V\bm p\cdot\bm n-\bm\beta\cdot\bm nP_Mz+\tau_2\widetilde{\Pi}_Wz,\mu\rangle_{e}&=\langle\bm p\cdot\bm n-\bm\beta\cdot\bm nz+\tau_2 z,\mu\rangle_{e},\label{projection_operator1_c}
	\end{align}
	for all  $(\bm{r} , w, \mu)\in \bm{\mathcal P}_{k-1}(K)\times {\mathcal P}_{k-1}(K)\times  {\mathcal P}_{k}(e)$ and for all faces $e$ of the simplex $K$.
\end{subequations}

Again, results from \cite{MR2991828} give the following estimates.
\begin{lemma}\label{pro_error1}
	Suppose $k\ge0$, and $ \tau_2 $ satisfies  \eqref{eqn:tau1_condition}. Then the system \eqref{projection_operator1} is uniquely solvable for $\widetilde{\bm{\Pi}}_V\bm{p}$ and $\widetilde{\Pi}_W z$, and
	\begin{subequations}
		\begin{align}
		\|\widetilde{\bm\Pi}_V\bm p-\bm p\|_{K}&\le Ch^{k+1}|\bm p|_{k+1,K}+Ch^{k+1}|z|_{k+1,K},\\
		\|\widetilde{\Pi}_Wz-z\|_{K}&\le Ch^{k+1}|\bm p|_{k+1,K}+Ch^{k+1}|z|_{k+1,K},
		\end{align}
	\end{subequations}
	where $C$ is a constant depending on the polynomial degree and the shape-regularity parameters of the elements.
\end{lemma}

Next, we present a basic approximation of the function $\bm \beta$. Let $ \bm P_0 $ be the vectorial piecewise-constant $ L^2 $ projection.  We have the following estimate:
\begin{align*}
\norm {\bm\beta - \bm P_0 \bm\beta}_{0,\infty,\Omega} \le C h \|\bm \beta\|_{1,\infty,\Omega}.
\end{align*}

\begin{lemma}\label{tau_lem}
	For any $e\in \partial K$, define $\widetilde \tau_2|_e = \tau_1|_e - \bm P_0 \bm\beta|_K\cdot \bm n_{e}$, we have 
	\begin{align*}
	\norm{\tau_2 - \widetilde \tau_2}_{0,\infty,\partial \mathcal T_h} \le C_{\bm \beta}h\|\bm \beta\|_{1,\infty,\Omega}.
	\end{align*}
\end{lemma}
\begin{proof}
	\begin{align*}
	\norm{\tau_2 - \widetilde \tau_2}_{0,\infty,\partial\mathcal T_h} &= \sum_{K\in\mathcal T_h}\norm{\tau_2 - \widetilde \tau_2}_{0,\infty,\partial K} \\
	& = \sum_{K\in\mathcal T_h}  \norm{\tau_1 - \bm\beta\cdot\bm n - \tau_1 + \bm P_0\bm \beta\cdot \bm n}_{0,\infty,\partial K}\\
	& =  \sum_{K\in\mathcal T_h}  \norm{\bm\beta\cdot\bm n -  \bm P_0\bm \beta\cdot \bm n}_{0,\infty, K}\\
	& \le  \norm {\bm\beta - \bm P_0 \bm\beta}_{0,\infty,\Omega}\\
	& \le Ch\|\bm \beta\|_{1,\infty,\Omega}.
	\end{align*}
\end{proof}

We define the following HDG operators $ \mathscr B_1$ and $ \mathscr B_2 $.
\begin{equation}\label{def_B1}
\begin{split}
\hspace{1em}&\hspace{-1em}  \mathscr  B_1( \bm q_h,y_h,\widehat y_h^o;\bm r_1,w_1,\mu_1) \\
&=(\bm q_h,\bm r_1)_{\mathcal T_h}-( y_h,\nabla\cdot\bm r_1)_{\mathcal T_h}+\langle \widehat y_h^o,\bm r_1\cdot\bm n\rangle_{\partial\mathcal T_h\backslash \varepsilon_h^\partial} \\
&\quad-(\bm q_h+\bm \beta y_h,  \nabla w_1)_{\mathcal T_h} +\langle {\bm q}_h\cdot\bm n +\tau_1y_h,w_1\rangle_{\partial\mathcal T_h}+\langle (\bm\beta\cdot\bm n -\tau_1) \widehat y_h^o,w_1\rangle_{\partial\mathcal T_h\backslash \varepsilon_h^\partial}\\
&\quad -\langle  {\bm q}_h\cdot\bm n+\bm \beta\cdot\bm n\widehat y_h^o +\tau_1(y_h-\widehat y_h^o),\mu_1\rangle_{\partial\mathcal T_h\backslash\varepsilon^{\partial}_h},\\
\hspace{1em}&\hspace{-1em}  \mathscr B_2(\bm p_h,z_h,\widehat z_h^o;\bm r_2, w_2,\mu_2)\\
&=(\bm p_h,\bm r_2)_{\mathcal T_h}-( z_h,\nabla\cdot\bm r_2)_{\mathcal T_h}+\langle \widehat z_h^o,\bm r_2\cdot\bm n\rangle_{\partial\mathcal T_h\backslash\varepsilon_h^\partial}\\
&\quad-(\bm p_h-\bm \beta z_h,  \nabla w_2)_{\mathcal T_h}
+\langle {\bm p}_h\cdot\bm n +\tau_2 z_h ,w_2\rangle_{\partial\mathcal T_h} -\langle (\bm \beta\cdot\bm n + \tau_2)\widehat z_h^o ,w_2\rangle_{\partial\mathcal T_h\backslash\varepsilon_h^\partial}
\\
&\quad-\langle  {\bm p}_h\cdot\bm n-\bm \beta\cdot\bm n\widehat z_h^o +\tau_2(z_h-\widehat z_h^o),\mu_2\rangle_{\partial\mathcal T_h\backslash\varepsilon^{\partial}_h}.
\end{split}
\end{equation}
By the definition in \eqref{def_B1},  we can rewrite the HDG formulation of the optimality system \eqref{HDG_discrete2} as follows: find $({\bm{q}}_h,{\bm{p}}_h,y_h,z_h,u_h,\widehat y_h^o,\widehat z_h^o)\in \bm{V}_h\times\bm{V}_h\times W_h \times W_h\times W_h\times M_h(o)\times M_h(o)$  such that
\begin{subequations}\label{HDG_full_discrete}
	\begin{align}
	\mathscr B_1(\bm q_h,y_h,\widehat y_h^o;\bm r_1,w_1,\mu_1)&=( f+ u_h, w_1)_{\mathcal T_h}\ \nonumber\\
	&\quad-\langle g, (\bm\beta\cdot\bm n-\tau_1)w_1+\bm r_1\cdot\bm n \rangle_{\varepsilon_h^\partial},\label{HDG_full_discrete_a}\\
	\mathscr B_2(\bm p_h,z_h,\widehat z_h^o;\bm r_2,w_2,\mu_2)&=(y_d-y_h,w_2)_{\mathcal T_h},\label{HDG_full_discrete_b}\\
	(z_h-\gamma u_h,w_3)_{\mathcal T_h}&= 0,\label{HDG_full_discrete_e}
	\end{align}
\end{subequations}
for all $\left(\bm{r}_1, \bm{r}_2,w_1,w_2,w_3,\mu_1,\mu_2\right)\in \bm{V}_h\times\bm{V}_h\times W_h \times W_h\times W_h\times M_h(o)\times M_h(o)$.

Next, we present a basic property of the operators $\mathscr B_1$ and $\mathscr B_2$,   and show the HDG equations \eqref{HDG_full_discrete} have a unique solution.
\begin{lemma}\label{property_B}
	For any $ ( \bm v_h, w_h, \mu_h ) \in \bm V_h \times W_h \times M_h(o) $, we have
	\begin{align*}
	\hspace{2em}&\hspace{-2em} \mathscr B_1(\bm v_h,w_h,\mu_h;\bm v_h,w_h,\mu_h)\\
	&=(\bm v_h,\bm v_h)_{\mathcal T_h}+ \langle (\tau_1 - \frac 12 \bm \beta\cdot\bm n)(w_h-\mu_h),w_h-\mu_h\rangle_{\partial\mathcal T_h\backslash \varepsilon_h^\partial}\\
	&\quad+\langle (\tau_1-\frac12\bm \beta\cdot\bm n) w_h,w_h\rangle_{\varepsilon_h^\partial},\\
	\hspace{2em}&\hspace{-2em}\mathscr B_2(\bm v_h,w_h,\mu_h;\bm v_h,w_h,\mu_h)\\
	&=(\bm v_h,\bm v_h)_{\mathcal T_h}+ \langle (\tau_2 + \frac 12 \bm \beta\cdot\bm n)(w_h-\mu_h),w_h-\mu_h\rangle_{\partial\mathcal T_h\backslash \varepsilon_h^\partial}\\
	&\quad+\langle (\tau_2+\frac12\bm \beta\cdot\bm n) w_h,w_h\rangle_{\varepsilon_h^\partial}.
	\end{align*}
\end{lemma}
\begin{proof}
	We only prove the first identity; the second can be obtained by the same argument.
	\begin{align*}
	\hspace{3em}&\hspace{-3em} 	\mathscr B_1(\bm v_h,w_h,\mu_h;\bm v_h,w_h,\mu_h)\\
	&=(\bm v_h,\bm v_h)_{\mathcal T_h}-( w_h,\nabla\cdot\bm v_h)_{\mathcal T_h}+\langle \mu_h,\bm v_h\cdot\bm n\rangle_{\partial\mathcal T_h\backslash \varepsilon_h^\partial}\\
	& \quad -(\bm v_h+\bm \beta w_h,  \nabla w_h)_{\mathcal T_h} +\langle {\bm v}_h\cdot\bm n +\tau_1 w_h,w_h\rangle_{\partial\mathcal T_h} \\
	&\quad+\langle (\bm\beta\cdot\bm n -\tau_1) \mu_h,w_h\rangle_{\partial\mathcal T_h\backslash \varepsilon_h^\partial}\\
	& \quad-\langle  {\bm v}_h\cdot\bm n+\bm \beta\cdot\bm n\mu_h +\tau_1(w_h - \mu_h),\mu_h \rangle_{\partial\mathcal T_h\backslash\varepsilon^{\partial}_h}\\
	&=(\bm v_h,\bm v_h)_{\mathcal T_h}-(\bm \beta w_h,  \nabla w_h)_{\mathcal T_h}+\langle  \tau_1 w_h,w_h\rangle_{\partial\mathcal T_h} \\
	&\quad+\langle (\bm\beta\cdot\bm n -\tau_1) \mu_h,w_h\rangle_{\partial\mathcal T_h\backslash \varepsilon_h^\partial} -\langle  \bm \beta\cdot\bm n \mu_h +\tau_1(w_h - \mu_h ),\mu_h\rangle_{\partial\mathcal T_h\backslash\varepsilon^{\partial}_h}.
	\end{align*}
	Moreover,
	\begin{align*}
	(\bm \beta w_h,\nabla w_h)_{\mathcal T_h}=(\nabla\cdot(\bm \beta w_h),w_h)_{\mathcal T_h} =\langle\bm \beta\cdot\bm n w_h,w_h\rangle_{\partial\mathcal T_h}-(\bm \beta w_h,\nabla w_h)_{\mathcal T_h},
	\end{align*}
	which implies
	\begin{align*}
	(\bm \beta w_h,\nabla w_h)_{\mathcal T_h}=\frac12\langle\bm \beta\cdot\bm n w_h,w_h\rangle_{\partial\mathcal T_h}.
	\end{align*}
	Then we obtain
	\begin{align*}
	\hspace{1em}&\hspace{-1em}  \mathscr B_1 (\bm v_h,w_h,\mu_h;\bm v_h,w_h,\mu_h)\\
	&=(\bm v_h,\bm v_h)_{\mathcal T_h}+ \langle ( \tau_1 - \frac 12 \bm \beta\cdot\bm n)(w_h-\mu_h),w_h-\mu_h\rangle_{\partial\mathcal T_h\backslash \varepsilon_h^\partial}\\
	&\quad + \langle (\tau_1-\frac12\bm \beta\cdot\bm n) w_h,w_h\rangle_{\varepsilon_h^\partial}-\frac12\langle\bm \beta\cdot\bm n \mu_h,\mu_h\rangle_{\partial\mathcal T_h\backslash \varepsilon_h^\partial}.
	\end{align*}
	Since $ \mu_h$ is single-valued across the interfaces, we have
	\begin{align*}
	-\frac12\langle\bm \beta\cdot\bm n\mu_h,\mu_h\rangle_{\partial\mathcal T_h\backslash\varepsilon_h^\partial}=0.
	\end{align*}
	This completes the proof.
\end{proof}

Next, we give a property of the HDG operators $\mathscr B_1$ and $\mathscr B_2$ that is critical to our error analysis of the method.
\begin{lemma}\label{identical_equa}
	If \textbf{(A2)} holds, then
	$$\mathscr B_1 (\bm q_h,y_h,\widehat y_h^o;\bm p_h,-z_h,-\widehat z_h^o) + \mathscr B_2 (\bm p_h,z_h,\widehat z_h^o;-\bm q_h,y_h,\widehat y_h^o) = 0.$$
\end{lemma}
\begin{proof}
	By the definition of $ \mathscr B_1 $  and $ \mathscr B_2 $,
	\begin{align*}
	\hspace{1em}&\hspace{-1em}  \mathscr B_1 (\bm q_h,y_h,\widehat y_h^o;\bm p_h,-z_h,-\widehat z_h^o) + \mathscr B_2 (\bm p_h,z_h,\widehat z_h^o;-\bm q_h,y_h,\widehat y_h^o)\\
	&=(\bm{q}_h, \bm p_h)_{{\mathcal{T}_h}}- (y_h, \nabla\cdot \bm p_h)_{{\mathcal{T}_h}}+\langle \widehat{y}_h^o, \bm p_h\cdot \bm{n} \rangle_{\partial{{\mathcal{T}_h}}\backslash {\varepsilon_h^{\partial}}} \\
	&\quad + (\bm{q}_h + \bm{\beta} y_h, \nabla z_h)_{{\mathcal{T}_h}}  - \langle\bm q_h\cdot\bm n +\tau_1 y_h , z_h \rangle_{\partial{{\mathcal{T}_h}}}  - \langle(\bm\beta\cdot \bm n-\tau_1)\widehat y_h^o, z_h \rangle_{\partial{{\mathcal{T}_h}}\backslash \varepsilon_h^{\partial}} \\
	&\quad+ \langle\bm q_h\cdot\bm n + \bm{\beta}\cdot\bm n \widehat y_h^o  +\tau_1 (y_h-\widehat y_h^o), \widehat z_h^o  \rangle_{\partial{{\mathcal{T}_h}}\backslash\varepsilon_h^{\partial}}\\
	&\quad-(\bm{p}_h, \bm q_h)_{{\mathcal{T}_h}}+ (z_h, \nabla\cdot \bm q_h)_{{\mathcal{T}_h}} -\langle \widehat{z}_h^o, \bm q_h \cdot \bm{n} \rangle_{\partial{{\mathcal{T}_h}}\backslash {\varepsilon_h^{\partial}}}\\
	&\quad - (\bm{p}_h - \bm{\beta} z_h, \nabla y_h)_{{\mathcal{T}_h}} +\langle\bm p_h\cdot\bm n +\tau_2 z_h , y_h \rangle_{\partial{{\mathcal{T}_h}}} - \langle (\bm{\beta}\cdot \bm n+\tau_2 ) \widehat z_h^o, y_h \rangle_{\partial{{\mathcal{T}_h}}\backslash \varepsilon_h^{\partial}}\\
	&\quad- \langle\bm p_h\cdot\bm n -\bm{\beta} \cdot\bm n\widehat z_h^o + \tau_2 (z_h-\widehat z_h^o), \widehat y_h^o \rangle_{\partial{{\mathcal{T}_h}}\backslash\varepsilon_h^{\partial}}.
	\end{align*}
	Integration by parts gives
	\begin{align*}
	\mathscr B_1 &(\bm q_h,y_h,\widehat y_h^o;\bm p_h,-z_h,-\widehat z_h^o) + \mathscr B_2 (\bm p_h,z_h,\widehat z_h^o;-\bm q_h,y_h,\widehat y_h^o)\\
	&=\langle (\tau_2 + \bm{\beta}\cdot\bm n-\tau_1) y_h, z_h \rangle_{\partial\mathcal T_h} + \langle (\tau_2 + \bm{\beta}\cdot\bm n-\tau_1) \widehat y_h^o, \widehat z_h^o \rangle_{\partial\mathcal T_h\backslash\varepsilon_h^\partial}.
	\end{align*}
	The proof is complete by assumption \textbf{(A2)}.
\end{proof}

\begin{proposition}\label{ex_uni}
	There exists a unique solution of the HDG equations \eqref{HDG_full_discrete}.
\end{proposition}
\begin{proof}
	Since the system \eqref{HDG_full_discrete} is finite dimensional, we only need to prove the uniqueness.  Therefore, we assume $y_d = f =g= 0$ and we show the system \eqref{HDG_full_discrete} only has the trivial solution.
	
	Take $(\bm r_1,w_1,\mu_1) = (\bm p_h,-z_h,-\widehat z_h^o)$, $(\bm r_2,w_2,\mu_2) = (-\bm q_h,y_h,\widehat y_h^o)$, and $w_3 = z_h -\gamma u_h $ in the HDG equations \eqref{HDG_full_discrete_a},  \eqref{HDG_full_discrete_b}, and \eqref{HDG_full_discrete_e}, respectively, and sum to obtain
	\begin{align*}
	\hspace{3em}&\hspace{-3em} \mathscr B_1  (\bm q_h,y_h,\widehat y_h^o;\bm p_h,-z_h,-\widehat z_h^o) + \mathscr B_2 (\bm p_h,z_h,\widehat z_h^o;-\bm q_h,y_h,\widehat y_h^o) \\
	& =	\gamma (y_h,y_h)_{\mathcal T_h} +   (z_h,z_h)_{\mathcal T_h}
	\end{align*}
	Since $\gamma>0$, Lemma \ref{identical_equa} implies $y_h =  u_h = z_h= 0$.
	
	Next, take $(\bm r_1,w_1,\mu_1) = (\bm q_h,y_h,\widehat y_h^o)$ and $(\bm r_2,w_2,\mu_2) = (\bm p_h,z_h,\widehat z_h^o)$ in the HDG equations \eqref{HDG_full_discrete_a}-\eqref{HDG_full_discrete_b}.  Lemma \eqref{property_B} and \textbf{(A2)} and \textbf{(A3)} give $\bm q_h= \bm p_h= \bm 0 $ and $ \widehat y_h^o  = \widehat z_h^o=0$.
\end{proof}

\subsection{Proof of Main Result}
To prove the main result, we follow the strategy of our earlier work \cite{HuShenSinglerZhangZheng_HDG_Dirichlet_control1} and split the proof into
five steps.  We consider the following auxiliary problem: find $$({\bm{q}}_h(u),{\bm{p}}_h(u), y_h(u), z_h(u), {\widehat{y}}_h^o(u), {\widehat{z}}_h^o(u))\in \bm{V}_h\times\bm{V}_h\times W_h \times W_h\times M_h(o)\times M_h(o)$$ such that
\begin{subequations}\label{HDG_inter_u}
	\begin{align}
	\mathscr B_1(\bm q_h(u),y_h(u),\widehat{y}_h(u);\bm r_1, w_1,\mu_1)&=( f+ u,w_1)_{\mathcal T_h} \ \nonumber\\
	& \quad- \langle g, (\bm\beta\cdot\bm n-\tau_1)w_1+\bm r_1\cdot\bm n \rangle_{\varepsilon_h^\partial},\label{HDG_u_a} \\
	\mathscr B_2(\bm p_h(u),z_h(u),\widehat{z}_h(u);\bm r_2, w_2,\mu_2)&=(y_d-y_h(u), w_2)_{\mathcal T_h},\label{HDG_u_b}
	\end{align}
\end{subequations}
for all $\left(\bm{r}_1, \bm{r}_2,w_1,w_2,\mu_1,\mu_2\right)\in \bm{V}_h\times\bm{V}_h \times W_h\times W_h\times M_h(o)\times M_h(o)$.  We begin by bounding the error between the solutions of the auxiliary problem and the mixed form \eqref{mixed_a}-\eqref{mixed_d} of the optimality system.

\subsubsection{Step 1: The error estimates for $\norm{ q- q_h(u)}_{\mathcal T_h}$ and $\norm {y-y_h(u)}_{\mathcal T_h}$.} \label{subsec:proof_step1}

The auxiliary HDG equation \eqref{HDG_u_a} is precisely the standard HDG discretization of the convection diffusion PDE \eqref{eq_adeq_a}-\eqref{eq_adeq_b} for $ y $ since the exact optimal control $ u $ is fixed in \eqref{HDG_u_a}.  The HDG error estimates for this problem have already been obtained in \cite{MR2991828}:
%
\begin{lemma}[\cite{MR2991828}]\label{le}
	If conditions {\bf{(A1)}} and {\bf{(A2)}} hold, we have
	\begin{align}\label{err_y_yhu}
	\|y-y_h(u)\|_{\mathcal T_h}+\|\bm q-\bm q_h(u)\|_{\mathcal T_h}\le Ch^{k+1}(|\bm q|_{k+1}+|y|_{k+1}).
	\end{align}
\end{lemma}

\subsubsection{Step 2: The error equation for part 2 of the auxiliary problem \eqref{HDG_u_b}.} \label{subsec:proof_step2}

Next, we bound the error between the solution of the dual convection diffusion equation \eqref{eq_adeq_c}-\eqref{eq_adeq_d} for $ z $ and the auxiliary HDG equation \eqref{HDG_u_b}.  We split the errors in the variables using the HDG  projections.  Define
\begin{equation}\label{notation_1}
\begin{split}
\delta^{\bm p} &=\bm p- \widetilde{\bm\Pi}_V\bm p,  \qquad\qquad\qquad \qquad\qquad\qquad\;\;\;\;\varepsilon^{\bm p}_h=\widetilde{\bm\Pi}_V \bm p-\bm p_h(u),\\
\delta^z&=z- \widetilde{\Pi}_W z, \qquad\qquad\qquad \qquad\qquad\qquad\;\;\;\; \;\varepsilon^{z}_h=\widetilde{\Pi}_W z-z_h(u),\\
\delta^{\widehat z} &= z-P_Mz,  \qquad\qquad\qquad\qquad\qquad\qquad\quad\;\; \varepsilon^{\widehat z}_h=P_M z-\widehat{z}_h(u),\\
\widehat {\bm\delta}_2 &= \delta^{\bm p}\cdot\bm n+\tau  \delta^z,  \qquad\qquad\qquad\qquad\quad\quad\quad\quad
\widehat {\bm \varepsilon }_2= \varepsilon_h^{\bm p}\cdot\bm n+\tau (\varepsilon^z_h-\varepsilon_h^{\widehat z}).
\end{split}
\end{equation}
where $\widehat z_h(u) = \widehat z_h^o(u)$ on $\varepsilon_h^o$ and $\widehat z_h(u) = 0$ on $\varepsilon_h^{\partial}$.  This gives $\varepsilon_h^{\widehat z} = 0$ on $\varepsilon_h^{\partial}$.

\begin{lemma}\label{lemma:step1_first_lemma}
	We have
	\begin{align}\label{error_z}
	\hspace{3em}&\hspace{-3em} \mathscr B_2(\varepsilon^{\bm p}_h,\varepsilon^z_h,\varepsilon^{\widehat{z}}_h;\bm r_2, w_2,\mu_2) \ \nonumber\\
	&=(\delta^{\bm p},\bm r_2)_{\mathcal T_h}+(y_h(u)- y, w_2)_{\mathcal T_h}+  \langle (\tau_2 - \widetilde \tau_2)\delta^{\widehat z},w_2 - \mu_2\rangle_{\partial\mathcal T_h}.
	\end{align}
\end{lemma}

\begin{proof}
	By the definition of operator $\mathscr B_2$ \eqref{def_B1}, we have
	\begin{align*}
	\hspace{3em}&\hspace{-3em} \mathscr B_2 (\widetilde{\bm \Pi}_V\bm p,\widetilde{\Pi}_W z,P_Mz;\bm r_2,w_2,\mu_2) \\
	& =(\widetilde{\bm \Pi}_V\bm p,\bm r_2)_{\mathcal T_h}-(\widetilde{\Pi}_W z,\nabla\cdot\bm r_2)_{\mathcal T_h}+\langle P_M z,\bm r_2\cdot\bm n\rangle_{\partial\mathcal T_h\backslash\varepsilon_h^\partial}\\
	&\quad-(\widetilde{\bm \Pi}_V\bm p -\bm \beta \widetilde{\Pi}_W z,  \nabla w_2)_{\mathcal T_h} +\langle \widetilde{\bm \Pi}_V\bm p \cdot\bm n +\tau_2  \widetilde{\Pi}_W z ,w_2\rangle_{\partial\mathcal T_h}\nonumber\\
	&\quad-\langle (\bm \beta\cdot\bm n + \tau_2)P_Mz ,w_2\rangle_{\partial\mathcal T_h\backslash\varepsilon_h^\partial}\\
	&\quad-\langle  \widetilde{\bm \Pi}_V\bm p \cdot\bm n-\bm \beta\cdot\bm nP_Mz+\tau_2( \widetilde{\Pi}_W z -P_Mz ),\mu_2\rangle_{\partial\mathcal T_h\backslash\varepsilon^{\partial}_h}.
	\end{align*}
	By properties of the HDG projections $\widetilde {\bm \Pi}_V$ and $\widetilde \Pi_W$ in \eqref{projection_operator1_c} and the $L^2 $ projection $P_M$ in \eqref{P_M}, we have
	\begin{align*}
	\langle\widetilde{\bm\Pi}_V\bm p\cdot\bm n+\tau_2\widetilde{\Pi}_Wz,w_2\rangle_{\partial\mathcal T_h}=\langle\bm p\cdot\bm n +\bm\beta\cdot\bm nP_Mz -\bm\beta\cdot\bm nz+\tau_2 z,w_2\rangle_{\partial\mathcal T_h},\\
	\langle\widetilde{\bm\Pi}_V\bm p\cdot\bm n-\bm\beta\cdot\bm nP_Mz+\tau_2\widetilde{\Pi}_Wz,\mu\rangle_{\partial \mathcal T_h\backslash \varepsilon_h^\partial}=\langle\bm p\cdot\bm n-\bm\beta\cdot\bm nz+\tau_2 z,\mu\rangle_{\partial \mathcal T_h\backslash \varepsilon_h^\partial}.
	\end{align*}
	By \eqref{projection_operator1_a}-\eqref{projection_operator1_b}, we have
	\begin{align*}
	\hspace{1em}&\hspace{-1em}  \mathscr B_2 (\widetilde{\bm \Pi}_V\bm p,\widetilde{\Pi}_W z,P_Mz;\bm r_2,w_2,\mu_2)\\
	& = (\bm p,\bm r_2)_{\mathcal T_h} - (\delta^{\bm p},\bm r_2)_{\mathcal T_h}-(z,\nabla\cdot\bm r_2)_{\mathcal T_h}+\langle  z,\bm r_2\cdot\bm n\rangle_{\partial\mathcal T_h\backslash\varepsilon_h^\partial}\\
	&\quad-(\bm p -\bm \beta z,  \nabla w_2)_{\mathcal T_h} +\langle \bm p\cdot\bm n  -\bm\beta\cdot\bm nz ,w_2\rangle_{\partial\mathcal T_h} + \langle \bm\beta\cdot\bm nP_Mz + \tau_2 z ,w_2\rangle_{\partial\mathcal T_h}\\
	&\quad-\langle (\bm \beta\cdot\bm n + \tau_2)P_Mz ,w_2\rangle_{\partial\mathcal T_h\backslash\varepsilon_h^\partial}
	-\langle  \bm p\cdot\bm n-\bm\beta\cdot\bm nz ,\mu_2\rangle_{\partial\mathcal T_h\backslash\varepsilon^{\partial}_h}\\
	& \quad- \langle  \tau_2 z -\tau_2 P_Mz ,\mu_2\rangle_{\partial\mathcal T_h\backslash\varepsilon^{\partial}_h}.
	\end{align*}
	Note  that the exact solution $\bm p$ and $z$ satisfies
	\begin{align*}
	(\bm p,\bm r_2)_{\mathcal T_h}-(z,\nabla\cdot\bm r_2)_{\mathcal T_h}+\langle  z,\bm r_2\cdot\bm n\rangle_{\partial\mathcal T_h}&=0,\\
	-(\bm p-\bm \beta z, \nabla w_2)_{\mathcal T_h}+\langle {\bm p}\cdot\bm n-\bm \beta\cdot\bm n z,w_2\rangle_{\partial\mathcal T_h}&= (y_d- y, w_2)_{\mathcal T_h},\\
	\langle {\bm p}\cdot\bm n-\bm \beta\cdot\bm n z,\mu_2\rangle_{\partial\mathcal T_h\backslash\varepsilon^{\partial}_h}&=0,
	\end{align*}
	for all $\left(\bm{r}_2, w_2,\mu_2\right)\in \bm V_h\times W_h\times M_h(o)$.  Since $z=0$ on $\varepsilon_h^\partial$, we have
	\begin{align*}
	\hspace{3em}&\hspace{-3em}  \mathscr B_2 (\widetilde{\bm \Pi}_V\bm p,\widetilde{\Pi}_W z,P_Mz;\bm r_2,w_2,\mu_2) \\
	& = - (\delta^{\bm p},\bm r_2)_{\mathcal T_h} + (y_d - y,w_2)_{\mathcal T_h} + \langle  \tau_2\delta^{\widehat z} ,w_2-\mu_2\rangle_{\partial\mathcal T_h}.
	\end{align*}
	By the definition of $P_M$ in \eqref{P_M} and since $\widetilde \tau_2$ from Lemma \ref{tau_lem} is piecewise constant on $\partial \mathcal T_h$, we have
	\begin{align*}
	\langle \tau_2 \delta^{\widehat z}, w_2-\mu_2\rangle_{\partial \mathcal T_h} &= \langle (\tau_2 - \widetilde \tau_2)\delta^{\widehat z}, w_2-\mu_2\rangle_{\partial \mathcal T_h}.
	\end{align*}
	This gives
	\begin{align*}
	\hspace{3em}&\hspace{-3em}  \mathscr B_2 (\widetilde{\bm \Pi}_V\bm p,\widetilde{\Pi}_W z,P_Mz;\bm r_2,w_2,\mu_2) \\
	& = - (\delta^{\bm p},\bm r_2)_{\mathcal T_h} + (y_d - y,w_2)_{\mathcal T_h} + \langle   (\tau_2 - \widetilde \tau_2) \delta^{\widehat z} ,w_2 - \mu_2\rangle_{\partial\mathcal T_h}.
	\end{align*}
	Subtract part 2 of the auxiliary problem \eqref{HDG_u_b} from the above equality to obtain the result.
\end{proof}

\subsubsection{Step 3: Estimates for $\varepsilon_h^p$ and $\varepsilon_h^z$ by an energy and duality argument.} \label{subsec:proof_step3}

\begin{lemma}\label{e_sec}
	We have
	\begin{align}
	\|\varepsilon_h^{\bm p}\|_{\mathcal T_h}&+\|\varepsilon_h^z-\varepsilon_h^{\widehat z}\|_{\partial\mathcal T_h} \le  \mathbb E + \kappa  \| \varepsilon^z_h \|_{\mathcal T_h},\label{error_es_p}
	\end{align}
	where
	\begin{align*}
	\mathbb E  =   C\|  \delta^{\bm p}  \|_{\mathcal T_h}   + \frac  C {\kappa} \| y_h(u) - y \|_{\mathcal T_h} +C\| \tau_2 - \widetilde \tau_2\|_{0, \infty, \partial\mathcal T_h} \| \delta^{\widehat z}\|_{\partial \mathcal T_h}
	\end{align*}
	and $ \kappa $ is any positive constant and $ C $ does not depend on $ \kappa $.
\end{lemma}

\begin{proof}
	Taking $(\bm r_2,w_2,\mu_2) = (\varepsilon^{\bm p}_h,\varepsilon^z_h,\varepsilon^{\widehat z}_h)$ in \eqref{error_z} in Lemma \ref{lemma:step1_first_lemma} gives
	\begin{align*}
	\hspace{1em}&\hspace{-1em}  \mathscr B_2 ( \varepsilon^{ \bm p}_h,  \varepsilon^z_h, \varepsilon^{\widehat z}_h;\varepsilon^{\bm p}_h, \varepsilon^z_h, \varepsilon^{\widehat z}_h )\\
	& = (\delta^{\bm p},\varepsilon^{\bm p}_h)_{\mathcal T_h}+(y_h(u)- y, \varepsilon^z_h)_{\mathcal T_h}+ \langle (\tau_2 - \widetilde \tau_2)\delta^{\widehat z},\varepsilon_h^z-\varepsilon_h^{\widehat z}\rangle_{\partial\mathcal T_h}\\
	&\le \|  \delta^{\bm p}  \|_{\mathcal T_h}  \| \varepsilon_ h^ {\bm p}  \|_{\mathcal T_h}  + \| y_h(u) - y \|_{\mathcal T_h} \| \varepsilon^z_h \|_{\mathcal T_h} \\
	&\quad+\| \tau_2 - \widetilde \tau_2\|_{0, \infty, \partial\mathcal T_h} \| \delta^{\widehat z}\|_{\partial \mathcal T_h} \|\varepsilon_h^z-\varepsilon_h^{\widehat z}\|_{\partial\mathcal T_h}.
	\end{align*}
	Lemma \eqref{property_B} gives
	\begin{align*}
	\|\varepsilon_h^{\bm p}\|_{\mathcal T_h}&+\|\varepsilon_h^z-\varepsilon_h^{\widehat z}\|_{\partial\mathcal T_h}\nonumber\\
	&\le  C\|  \delta^{\bm p}  \|_{\mathcal T_h}   + \frac  C {\kappa} \| y_h(u) - y \|_{\mathcal T_h} +C\| \tau_2 - \widetilde \tau_2\|_{0, \infty, \partial\mathcal T_h} \| \delta^{\widehat z}\|_{\partial \mathcal T_h} + \kappa  \| \varepsilon^z_h \|_{\mathcal T_h},
	\end{align*}
	where $\kappa$ is any positive constant.
\end{proof}

Next, we introduce the dual problem for any given $\Theta$ in $L^2(\Omega):$
\begin{equation}\label{Dual_PDE}
\begin{split}
\bm\Phi-\nabla\Psi&=0\qquad~\text{in}\ \Omega,\\
-\nabla\cdot\bm{\Phi}+\bm\beta\cdot\nabla\Psi&=\Theta \qquad\text{in}\ \Omega,\\
\Psi&=0\qquad~\text{on}\ \partial\Omega.
\end{split}
\end{equation}
Since the domain $\Omega$ is convex, we have the following regularity estimate
\begin{align}\label{dual_esti}
\norm{\bm \Phi}_{1,\Omega} + \norm{\Psi}_{2,\Omega} \le C_{\text{reg}} \norm{\Theta}_\Omega,
\end{align}

Before we estimate  $\varepsilon_h^{\bm p}$ and $\varepsilon_h^z$,  we introduce the following notation, which is similar to the earlier notation in \eqref{notation_1}:
\begin{align}
\delta^{\bm \Phi} &=\bm \Phi-\widetilde{\bm\Pi}_V\bm \Phi, \quad \delta^\Psi=\Psi- \widetilde{\Pi}_W \Psi, \quad
\delta^{\widehat \Psi} = \Psi-P_M\Psi.
\end{align}
\begin{lemma}
	We have
	\begin{subequations}
		\begin{align}
		\norm{\varepsilon_h^{\bm p}}_{\mathcal T_h}	&\lesssim h^{k+1}(|\bm q|_{k+1}+|y|_{k+1}+|\bm p|_{k+1}+|z|_{k+1}),\label{var_p}\\
		\|\varepsilon^z_h\|_{\mathcal T_h} &\lesssim h^{k+1}(|\bm q|_{k+1}+|y|_{k+1}+|\bm p|_{k+1}+|z|_{k+1}).\label{var_z}
		\end{align}
	\end{subequations}
\end{lemma}

\begin{proof}
	Consider the dual problem \eqref{Dual_PDE} and let $\Theta = \varepsilon_h^z$.  Take  $(\bm r_2,w_2,\mu_2) = (\widetilde {\bm\Pi}_V\bm{\Phi},\widetilde{\Pi}_W\Psi,P_M\Psi)$ in \eqref{error_z} in Lemma \ref{lemma:step1_first_lemma}.  Since $\Psi=0$ on $\varepsilon_h^{\partial}$ we have
	\begin{align*}
	\hspace{1em}&\hspace{-1em}  \mathscr B_2 (\varepsilon^{\bm p}_h,\varepsilon^z_h,\varepsilon^{\widehat z}_h;\widetilde{\bm\Pi}_V\bm{\Phi},\widetilde{\Pi}_W\Psi,P_M\Psi)\\
	&= (\varepsilon^{\bm p}_h,\widetilde{\bm\Pi}_V\bm{\Phi})_{\mathcal T_h}-( \varepsilon^z_h,\nabla\cdot\widetilde{\bm\Pi}_V\bm{\Phi})_{\mathcal T_h}+\langle \varepsilon^{\widehat z}_h,\widetilde{\bm\Pi}_V\bm{\Phi} \cdot\bm n\rangle_{\partial\mathcal T_h}\\
	&\quad-(\varepsilon^{\bm p}_h-\bm \beta \varepsilon^z_h,  \nabla \widetilde{\Pi}_W\Psi)_{\mathcal T_h}
	+\langle  \varepsilon^{\bm p}_h\cdot\bm n-\bm \beta\cdot\bm n\varepsilon^{\widehat z}_h +\tau_2(\varepsilon^z_h-\varepsilon^{\widehat z}_h ), \widetilde{\Pi}_W\Psi -P_M \Psi\rangle_{\partial\mathcal T_h}\\
	&=(\varepsilon^{\bm p}_h,\bm{\Phi})_{\mathcal{T}_h}-(\varepsilon^{\bm p}_h,\delta^{\bm \Phi})_{\mathcal{T}_h}-(\varepsilon^z_h,\nabla\cdot\bm{\Phi})_{\mathcal{T}_h}+(\varepsilon^z_h,\nabla\cdot \delta ^{\bm{\Phi}})_{\mathcal{T}_h} 
	\\
	&\quad-\langle \varepsilon^{\widehat z}_h, \delta^{\bm \Phi}\cdot \bm{n}\rangle_{\partial\mathcal{T}_h\backslash\varepsilon_h^\partial} -(\varepsilon^{\bm p}_h-\bm \beta \varepsilon^z_h, \nabla \Psi)_{\mathcal{T}_h}+(\varepsilon^{\bm p}_h-\bm \beta \varepsilon^z_h, \nabla \delta^{\Psi})_{\mathcal{T}_h}\\
	&\quad - \langle  \varepsilon^{\bm p}_h\cdot\bm n-\bm \beta\cdot\bm n\varepsilon^{\widehat z}_h +\tau_2(\varepsilon^z_h-\varepsilon^{\widehat z}_h ),\delta^{\Psi} - \delta^{\widehat \Psi}\rangle_{\partial\mathcal T_h}\\
	&=  -(\varepsilon^{\bm p}_h,\delta^{\bm \Phi})_{\mathcal{T}_h} + \|  \varepsilon_h^z\|_{\mathcal T_h}^2 + (\varepsilon^z_h,\nabla\cdot \delta ^{\bm{\Phi}})_{\mathcal{T}_h} -\langle \varepsilon^{\widehat z}_h, \delta^{\bm \Phi}\cdot \bm{n}\rangle_{\partial\mathcal{T}_h}\\
	& \quad+(\varepsilon^{\bm p}_h-\bm \beta \varepsilon^z_h, \nabla \delta^{\Psi})_{\mathcal{T}_h} -\langle  \varepsilon^{\bm p}_h\cdot\bm n-\bm \beta\cdot\bm n\varepsilon^{\widehat z}_h +\tau_2(\varepsilon^z_h-\varepsilon^{\widehat z}_h ),\delta^{\Psi} - \delta^{\widehat \Psi}\rangle_{\partial\mathcal T_h}.
	\end{align*}
	Here, we used $\langle\varepsilon^{\widehat z}_h,\bm \Phi\cdot\bm n\rangle_{\partial\mathcal T_h}=0$, which holds since  $\varepsilon^{\widehat z}_h$ is single-valued function on interior edges and $\varepsilon^{\widehat z}_h=0$ on $\varepsilon^{\partial}_h$. 
	
	Next, integration by parts gives
	\begin{align*}
	(\varepsilon^z_h,\nabla\cdot\delta^{\bm \Phi})_{\mathcal{T}_h}
	&=\langle \varepsilon^z_h,\delta^{\bm \Phi} \cdot\bm n\rangle_{\partial\mathcal T_h}-(\nabla\varepsilon^z_h,\delta^{\bm \Phi})_{\mathcal{T}_h} = \langle \varepsilon^z_h,\delta^{\bm \Phi}\cdot\bm n\rangle_{\partial\mathcal T_h}-(\nabla\varepsilon^z_h,\bm{\beta}\delta^{\Psi})_{\mathcal T_h},\\
	(\varepsilon^{\bm p}_h, \nabla \delta^{ \Psi})_{\mathcal{T}_h}&=\langle \varepsilon^{\bm p}_h \cdot\bm n, \delta^{ \Psi}\rangle_{\partial\mathcal T_h}-(\nabla\cdot \varepsilon^{\bm p}_h , \delta^{ \Psi})_{\mathcal T_h} = \langle \varepsilon^{\bm p}_h \cdot\bm n, \delta^{ \Psi}\rangle_{\partial\mathcal T_h},\\
	(\bm\beta \varepsilon_h^z, \nabla \delta^{ \Psi})_{\mathcal{T}_h}&=\langle \bm{\beta} \cdot\bm n \varepsilon_h^z, \delta^{ \Psi}\rangle_{\partial\mathcal T_h}  - (\bm{\beta} \nabla\varepsilon_h^z, \delta^{ \Psi})_{\mathcal T_h}. 
	\end{align*}
	We have 
	\begin{align*}
	\hspace{3em}&\hspace{-3em} \mathscr B_2(\varepsilon^{\bm p}_h,\varepsilon^z_h,\varepsilon^{\widehat z}_h;\widetilde{\bm\Pi}_V\bm{\Phi},\widetilde{\Pi}_W\Psi,P_M\Psi)\\
	& = -(\varepsilon^{\bm p}_h,\delta^{\bm \Phi})_{\mathcal{T}_h} + \|  \varepsilon_h^z\|_{\mathcal T_h}^2 + \langle \varepsilon^z_h,\delta^{\bm \Phi}\cdot\bm n\rangle_{\partial\mathcal T_h} -\langle \varepsilon^{\widehat z}_h, \delta^{\bm \Phi}\cdot \bm{n}\rangle_{\partial\mathcal{T}_h\backslash\varepsilon_h^\partial}
	\\
	&\quad-\langle \bm{\beta} \cdot\bm n \varepsilon_h^z, \delta^{ \Psi}\rangle_{\partial\mathcal T_h} -\langle -\bm \beta\cdot\bm n\varepsilon^{\widehat z}_h +\tau_2(\varepsilon^z_h-\varepsilon^{\widehat z}_h ),\delta^{\Psi} - \delta^{\widehat \Psi}\rangle_{\partial\mathcal T_h}.
	\end{align*}
	Remembering that $\varepsilon^{\widehat z}_h$ is single-valued function on interior edges and $\varepsilon^{\widehat z}_h=0$ on $\varepsilon^{\partial}_h$ gives
	\begin{align*}
	\langle\bm \beta\cdot\bm n\varepsilon^{\widehat z}_h,P_M\Psi\rangle_{\partial\mathcal T_h}=0=\langle\bm \beta\cdot\bm n\varepsilon^{\widehat z}_h,\Psi\rangle_{\partial\mathcal T_h}.
	\end{align*}
	This implies
	\begin{align*}
	\hspace{3em}&\hspace{-3em}  \mathscr B_2 (\varepsilon^{\bm p}_h,\varepsilon^z_h,\varepsilon^{\widehat z}_h;\widetilde{\bm\Pi}_V\bm{\Phi},\widetilde{\Pi}_W\Psi,P_M\Psi)\\
	&=-(\varepsilon^{\bm p}_h,\delta^{\bm \Phi})_{\mathcal{T}_h} + \|  \varepsilon_h^z\|_{\mathcal T_h}^2 + \langle \varepsilon^z_h - \varepsilon^{\widehat z}_h,\delta^{\bm \Phi}\cdot\bm n\rangle_{\partial\mathcal T_h} 
	\\
	&\quad-\langle \bm{\beta} \cdot\bm n (\varepsilon_h^z - \varepsilon_h^{\widehat z}), \delta^{ \Psi}\rangle_{\partial\mathcal T_h\backslash\varepsilon_h^\partial} -\langle \tau_2(\varepsilon^z_h-\varepsilon^{\widehat z}_h ),\delta^{\Psi} - \delta^{\widehat \Psi}\rangle_{\partial\mathcal T_h}\\
	&=-(\varepsilon^{\bm p}_h,\delta^{\bm \Phi})_{\mathcal{T}_h} + \|  \varepsilon_h^z\|_{\mathcal T_h}^2 + \langle \varepsilon^z_h - \varepsilon^{\widehat z}_h,\delta^{\bm \Phi}\cdot\bm n - \bm{\beta}\cdot\bm n\delta^{\Psi}-\tau_2(\delta^{\Psi} - \delta^{\widehat \Psi})\rangle_{\partial\mathcal T_h}.
	\end{align*}
	On the other hand,
	\begin{align*}
	\mathscr B_2 &(\varepsilon^{\bm p}_h,\varepsilon^z_h,\varepsilon^{\widehat z}_h;\widetilde{\bm\Pi}_V\bm{\Phi},\widetilde{\Pi}_W\Psi,P_M\Psi)\\
	& = (\delta^{\bm p},\widetilde{\bm\Pi}_V\bm{\Phi} )_{\mathcal T_h}+(y_h(u)- y, \widetilde{\Pi}_W\Psi)_{\mathcal T_h}+  \langle (\tau_2 - \widetilde \tau_2)\delta^{\widehat z},\widetilde{\Pi}_W\Psi - P_M\Psi\rangle_{\partial\mathcal T_h}.
	\end{align*}
	Comparing the above two equalities gives
	\begin{align*}
	\|  \varepsilon_h^z\|_{\mathcal T_h}^2 & = (\varepsilon^{\bm p}_h,\delta^{\bm \Phi})_{\mathcal{T}_h} -\langle \varepsilon^z_h - \varepsilon^{\widehat z}_h,\delta^{\bm \Phi}\cdot\bm n - \bm{\beta}\cdot\bm n\delta^{\Psi}-\tau_2(\delta^{\Psi} - \delta^{\widehat \Psi})\rangle_{\partial\mathcal T_h}\\
	&\quad+(\delta^{\bm p},\widetilde{\bm\Pi}_V\bm{\Phi} )_{\mathcal T_h}+(y_h(u)- y, \widetilde{\Pi}_W\Psi)_{\mathcal T_h}+  \langle (\tau_2 - \widetilde \tau_2) \delta^{\widehat z},\widetilde{\Pi}_W\Psi - P_M\Psi\rangle_{\partial\mathcal T_h}\\
	&=  \sum_{i=1}^7 R_i.
	\end{align*}
	Let $ C_0 = \max\{C, 1\} $, where $ C $ is the constant defined in Lemma \ref{pro_error1}. For the terms $R_1$ and $R_2$, Lemma \ref{e_sec} gives
	\begin{align*}
	R_1&=-(\varepsilon^{\bm p}_h,\delta^{\bm \Phi})_{\mathcal{T}_h}\le \|\varepsilon^{\bm p}_h\|_{\mathcal{T}_h} \|\delta^{\bm \Phi}\|_{\mathcal{T}_h}
	\le \left( \mathbb E+\kappa\|\varepsilon^z_h\|_{\mathcal T_h}\right) C_0(\|\bm \Phi\|_1 + \| \Psi\|_1)\\
	&\le C_0C_{\text{reg}}\left( \mathbb E+\kappa\|\varepsilon^z_h\|_{\mathcal T_h}\right) \|\varepsilon^z_h\|_{\mathcal T_h},\\
	R_2 &= -\langle \varepsilon^z_h - \varepsilon^{\widehat z}_h,\delta^{\bm \Phi}\cdot\bm n - \bm{\beta}\cdot\bm n\delta^{\Psi}-\tau_2(\delta^{\Psi} - \delta^{\widehat \Psi})\rangle_{\partial\mathcal T_h}\\
	&\le \| {\varepsilon^z_h - \varepsilon^{\widehat z}_h}\|_{\partial \mathcal T_h}(\|{\delta^{\bm \Phi}}\|_{\partial \mathcal T_h}+\| {\tau_1}\|_{0,\infty,\partial\mathcal T_h}\|{\delta^{\Psi}}\|_{\partial \mathcal T_h} +\| {\tau_2}\|_{0,\infty,\partial\mathcal T_h}\|{\delta^{\widehat\Psi}}\|_{\partial \mathcal T_h})\\
	&\le 3\left( \mathbb E+\kappa\|\varepsilon^z_h\|_{\mathcal T_h}\right)(1+\norm {\tau_1}_{0,\infty,\partial\mathcal T_h} + \norm {\tau_2}_{0,\infty,\partial\mathcal T_h})C_0(\|\bm \Phi\|_1 + \| \Psi\|_1)\\
	&\le 3C_0C_{\text{reg}}\left( \mathbb E+\kappa\|\varepsilon^z_h\|_{\mathcal T_h}\right)(1+\norm {\tau_1}_{0,\infty,\partial\mathcal T_h} + \norm {\tau_2}_{0,\infty,\partial\mathcal T_h})\|\varepsilon^z_h\|_{\mathcal T_h}.
	\end{align*}
	For the terms $R_3$, $R_4$ and $R_5$, we use the triangle inequality, the regularity estimate \eqref{dual_esti}, and the assumption $ h \leq 1 $ to give
	\begin{align*}
	R_3&=(\delta^{\bm p},\widetilde{\bm\Pi}_V\bm \Phi)_{\mathcal T_h}\le \|\delta^{\bm p}\|_{\mathcal T_h}(\|\widetilde{\bm\Pi}_V\bm \Phi- \bm \Phi\|_{\mathcal T_h} + \|\bm \Phi\|_{\mathcal T_h})\\
	&\le C_0\|\delta^{\bm p}\|_{\mathcal T_h}(\norm {\bm{\Phi}}_{1,\Omega} + \norm {{\Psi} }_{1,\Omega} +\|\bm \Phi\|_{\mathcal T_h})\\
	&\le 2C_0C_{\text{reg}}\|\delta^{\bm p}\|_{\mathcal T_h}\|\varepsilon^z_h\|_{\mathcal T_h},\\
	R_4&=(y-y_h(u),\widetilde\Pi_W\Psi)_{\mathcal T_h}\le \|y-y_h(u)\|_{\mathcal T_h} \|\widetilde\Pi_W\Psi\|_{\mathcal T_h}\\
	&\le \|y-y_h(u)\|_{\mathcal T_h}(\|\widetilde\Pi_W\Psi-\Psi\|_{\mathcal T_h}+\|\Psi\|_{\mathcal T_h})\\
	&\le C_0\|y-y_h(u)\|_{\mathcal T_h}(\|\Psi\|_{1,\Omega}+\|\bm\Phi\|_{1,\Omega}+\|\Psi\|_{\mathcal T_h})\\
	&\le 2C_0C_{\text{reg}}\|y-y_h(u)\|_{\mathcal T_h}\|\varepsilon^z_h\|_{\mathcal T_h},\\
	R_5 &= \langle (\tau_2 - \widetilde \tau_2) \delta^{\widehat z}, \widetilde{\Pi}_W\Psi - P_M\Psi\rangle_{\partial\mathcal T_h} \\
	&\le \| {\tau_2 - \widetilde \tau_2}\|_{0,\infty,\partial\mathcal T_h} \|{\delta^{\widehat z}}\|_{\partial\mathcal T_h} \| {\delta^{\Psi} - \delta^{\widehat \Psi}}\|_{\partial\mathcal T_h}\\
	&\le  C_{\bm \beta}\|\bm \beta\|_{1,\infty,\Omega}  h^{1/2}\|{\delta^{\widehat z}}\|_{\partial\mathcal T_h}C_0(\|\Psi\|_{1,\Omega}+\|\bm\Phi\|_{1,\Omega}+\|\Psi\|_{1,\Omega})\\
	&\le 2C_0C_{\text{reg}}  C_{\bm \beta}\|\bm \beta\|_{1,\infty,\Omega}  h^{1/2}\|{\delta^{\widehat z}}\|_{\partial\mathcal T_h} \|\varepsilon^z_h\|_{\mathcal T_h}.
	\end{align*}
	Summing $R_1$ to $R_5$ gives
	\begin{align*}
	\|\varepsilon^z_h\|_{\mathcal T_h} &\le\mathbb C(\mathbb E +\kappa\|\varepsilon^z_h\|_{\mathcal T_h} ) + C ( \norm {\delta^{\bm p}}_{\mathcal T_h} +\norm{y - y_h(u)}_{\mathcal T_h} + h^{1/2}\|{\delta^{\widehat z}}\|_{\partial\mathcal T_h}),
	\end{align*}
	where
	\begin{align*}
	\mathbb C =4 C_0C_{\text{reg}}  (1+\norm {\tau_1}_{0,\infty,\partial\mathcal T_h} + \norm {\tau_2}_{0,\infty,\partial\mathcal T_h}).
	\end{align*}
	Choose $\kappa=\frac{1}{2\mathbb C}$ gives
	\begin{align*}
	\|\varepsilon^z_h\|_{\mathcal T_h}\lesssim h^{k+1}(|\bm q|_{k+1}+|y|_{k+1}+|\bm p|_{k+1}+|z|_{k+1}).
	\end{align*}
	Finally, \eqref{error_es_p} and \eqref{var_z} imply  \eqref{var_p}.
\end{proof}
As a consequence, a simple application of the triangle inequality gives optimal convergence rates for $\|\bm p -\bm p_h(u)\|_{\mathcal T_h}$ and $\|z -z_h(u)\|_{\mathcal T_h}$:

\begin{lemma}\label{lemma:step3_conv_rates}
	\begin{subequations}
		\begin{align}
		\|\bm p -\bm p_h(u)\|_{\mathcal T_h} &\le \|\delta^{\bm p}\|_{\mathcal T_h} + \|\varepsilon_h^{\bm p}\|_{\mathcal T_h} \ \nonumber \\ 
		&\lesssim h^{k+1}(|\bm q|_{k+1}+|y|_{k+1}+|\bm p|_{k+1}+|z|_{k+1}),\\
		\|z -z_h(u)\|_{\mathcal T_h} &\le \|\delta^{z}\|_{\mathcal T_h} + \|\varepsilon_h^{z}\|_{\mathcal T_h} \ \nonumber \\
		& \lesssim  h^{k+1}(|\bm q|_{k+1}+|y|_{k+1}+|\bm p|_{k+1}+|z|_{k+1}).
		\end{align}
	\end{subequations}
\end{lemma}

\subsubsection{Step 4: Estimate for $\|u-u_h\|_{\mathcal T_h}$, $\norm {y-y_h}_{\mathcal T_h}$ and $\norm {z-z_h}_{\mathcal T_h}$.}

Next, we bound the error between the solutions of the auxiliary problem and the HDG discretization of the optimality system \eqref{HDG_full_discrete}.  We use these error bounds and the error bounds in Lemmas \ref{le} and \ref{lemma:step3_conv_rates} to obtain the main result.

For the remaining steps, we denote 
\begin{equation*}
\begin{split}
\zeta_{\bm q} &=\bm q_h(u)-\bm q_h,\quad\zeta_{y} = y_h(u)-y_h,\quad\zeta_{\widehat y} = \widehat y_h(u)-\widehat y_h,\\
\zeta_{\bm p} &=\bm p_h(u)-\bm p_h,\quad\zeta_{z} = z_h(u)-z_h,\quad\zeta_{\widehat z} = \widehat z_h(u)-\widehat z_h.
\end{split}
\end{equation*}
Subtracting the auxiliary problem and the HDG problem gives the following error equations
\begin{subequations}\label{eq_yh}
	\begin{align}
	\mathscr B_1(\zeta_{\bm q},\zeta_y,\zeta_{\widehat y};\bm r_1, w_1,\mu_1)&=(u-u_h,w_1)_{\mathcal T_h}\label{eq_yh_yhu}\\
	\mathscr B_2(\zeta_{\bm p},\zeta_z,\zeta_{\widehat z};\bm r_2, w_2,\mu_2)&=-(\zeta_y, w_2)_{\mathcal T_h}\label{eq_zh_zhu}.
	\end{align}
\end{subequations}
\begin{lemma}
	We have
	\begin{align}\label{eq_uuh_yhuyh}
	\hspace{3em}&\hspace{-3em}  \gamma\|u-u_h\|^2_{\mathcal T_h}+\|y_h(u)-y_h\|^2_{\mathcal T_h}\nonumber\\
	&=( z_h-\gamma u_h,u-u_h)_{\mathcal T_h}-(z_h(u)-\gamma u,u-u_h)_{\mathcal T_h}.
	\end{align}
\end{lemma}
\begin{proof}
	First, we have
	\begin{align*}
	\hspace{3em}&\hspace{-3em}  ( z_h-\gamma u_h,u-u_h)_{\mathcal T_h}-( z_h(u)-\gamma u,u-u_h)_{\mathcal T_h}\\
	&=-(\zeta_{ z},u-u_h)_{\mathcal T_h}+\gamma\|u-u_h\|^2_{\mathcal T_h}.
	\end{align*}
	Next, Lemma \ref{identical_equa} gives
	\begin{align*}
	\mathscr B_1 &(\zeta_{\bm q},\zeta_y,\zeta_{\widehat{y}};\zeta_{\bm p},-\zeta_{z},-\zeta_{\widehat z}) + \mathscr B_2(\zeta_{\bm p},\zeta_z,\zeta_{\widehat z};-\zeta_{\bm q},\zeta_y,\zeta_{\widehat{y}})  = 0.
	\end{align*}
	On the other hand, using the definition of $ \mathscr B_1 $ and $ \mathscr B_2 $ gives 
	\begin{align*}
	\hspace{3em}&\hspace{-3em}  \mathscr B_1 (\zeta_{\bm q},\zeta_y,\zeta_{\widehat{y}};\zeta_{\bm p},-\zeta_{z},-\zeta_{\widehat z}) + \mathscr B_2(\zeta_{\bm p},\zeta_z,\zeta_{\widehat z};-\zeta_{\bm q},\zeta_y,\zeta_{\widehat{y}})\\
	&= - ( u- u_h,\zeta_{ z})_{\mathcal{T}_h}-\|\zeta_{ y}\|^2_{\mathcal{T}_h}.
	\end{align*}
	Comparing the above two equalities gives
	\begin{align*}
	-(u-u_h,\zeta_{ z})_{\mathcal{T}_h}=\|\zeta_{ y}\|^2_{\mathcal{T}_h}.
	\end{align*}
	This completes the proof.
\end{proof}

\begin{theorem}\label{thm:estimates_u_y_z}
	We have
	\begin{subequations}
		\begin{align}\label{err_yhu_yh}
		\|u-u_h\|_{\mathcal T_h}&\lesssim h^{k+1}(|\bm q|_{k+1}+|y|_{k+1}+|\bm p|_{k+1}+|z|_{k+1}),\\
		\|y-y_h\|_{\mathcal T_h}&\lesssim h^{k+1}(|\bm q|_{k+1}+|y|_{k+1}+|\bm p|_{k+1}+|z|_{k+1}),\\
		\|z-z_h\|_{\mathcal T_h}&\lesssim h^{k+1}(|\bm q|_{k+1}+|y|_{k+1}+|\bm p|_{k+1}+|z|_{k+1}).
		\end{align}
	\end{subequations}
\end{theorem}
\begin{proof}
	Recall the continuous and discretized optimality conditions \eqref{eq_adeq_e} and \eqref{HDG_full_discrete_e} gives $ \gamma u = z $ and $ \gamma u_h = z_h $.  These equations and the previous lemma give
	\begin{align*}
	\hspace{3em}&\hspace{-3em}   \gamma\|u-u_h\|^2_{\mathcal T_h}+\|\zeta_{ y}\|^2_{\mathcal T_h}\\
	&=( z_h-\gamma u_h,u-u_h)_{\mathcal T_h}-( z_h(u)-\gamma u,u-u_h)_{\mathcal T_h}\\
	&=-( z_h(u)- z,u-u_h)_{\mathcal T_h}\\
	&\le \| z_h(u)- z\|_{\mathcal T_h} \|u-u_h\|_{\mathcal T_h}\\
	&\le\frac{1}{2\gamma}\| z_h(u)- z\|^2_{\mathcal T_h}+\frac{\gamma}{2}\|u-u_h\|^2_{\mathcal T_h}.
	\end{align*}
	By Lemma \ref{lemma:step3_conv_rates}, we have
	\begin{align}\label{eqn:estimate_u_zeta_y}
	\|u-u_h\|_{\mathcal T_h}+\|\zeta_{ y}\|_{\mathcal T_h}&\lesssim h^{k+1}(|\bm q|_{k+1}+|y|_{k+1}+|\bm p|_{k+1}+|z|_{k+1}).
	\end{align}
	Then, by the triangle inequality and Lemma \ref{le} we obtain
	\begin{align*}
	\|y-y_h\|_{\mathcal T_h}&\lesssim h^{k+1}(|\bm q|_{k+1}+|y|_{k+1}+|\bm p|_{k+1}+|z|_{k+1}).
	\end{align*}
	Finally, since $z = \gamma u $ and $z_h = \gamma u_h$ we have
	\begin{align*}
	\|z-z_h\|_{\mathcal T_h}&\lesssim h^{k+1}(|\bm q|_{k+1}+|y|_{k+1}+|\bm p|_{k+1}+|z|_{k+1}).
	\end{align*}
\end{proof}

\subsubsection{Step 5: Estimate for $\|q-q_h\|_{\mathcal T_h}$ and  $\|p-p_h\|_{\mathcal T_h}$.}

\begin{lemma}
	We have
	\begin{subequations}
		\begin{align}
		\|\zeta_{\bm q}\|_{\mathcal T_h} &\lesssim h^{k+1}(|\bm q|_{k+1}+|y|_{k+1}+|\bm p|_{k+1}+|z|_{k+1}),\label{err_Lhu_qh}\\
		\|\zeta_{\bm p}\|_{\mathcal T_h} &\lesssim h^{k+1}(|\bm q|_{k+1}+|y|_{k+1}+|\bm p|_{k+1}+|z|_{k+1}).\label{err_Lhu_ph}
		\end{align}
	\end{subequations}
\end{lemma}
\begin{proof}
	By Lemma \ref{property_B}, the error equation \eqref{eq_yh_yhu}, and the estimate \eqref{eqn:estimate_u_zeta_y} we have
	\begin{align*}
	\|\zeta_{\bm q}\|^2_{\mathcal T_h} &\lesssim  \mathscr B_1(\zeta_{\bm q},\zeta_y,\zeta_{\widehat y};\zeta_{\bm q},\zeta_y,\zeta_{\widehat y})\\
	&=( u- u_h,\zeta_{ y})_{\mathcal T_h}\\
	&\le\| u- u_h\|_{\mathcal T_h}\|\zeta_{ y}\|_{\mathcal T_h}\\
	&\lesssim h^{2k+2}(|\bm q|_{k+1}+|y|_{k+1}+|\bm p|_{k+1}+|z|_{k+1})^2.
	\end{align*}
	Similarly,  by Lemma \ref{property_B}, the error equation \eqref{eq_zh_zhu}, Lemma \ref{lemma:step3_conv_rates}, and Theorem \ref{thm:estimates_u_y_z} we have
	\begin{align*}
	\|\zeta_{\bm p}\|^2_{\mathcal T_h} &\lesssim  \mathscr B_2(\zeta_{\bm p},\zeta_z,\zeta_{\widehat z};\zeta_{\bm p},\zeta_z,\zeta_{\widehat z})\\
	&=-(\zeta_{ y},\zeta_{ z})_{\mathcal T_h}\\
	&\le\|\zeta_{y}\|_{\mathcal T_h}\|\zeta_{ z}\|_{\mathcal T_h}\\
	&\le\|\zeta_{y}\|_{\mathcal T_h} ( \| z_h(u) - z \|_{\mathcal T_h} + \| z - z_h \|_{\mathcal T_h} )\\
	&\lesssim h^{2k+2}(|\bm q|_{k+1}+|y|_{k+1}+|\bm p|_{k+1}+|z|_{k+1})^2.
	\end{align*}
\end{proof}
The above lemma along with the triangle inequality, Lemma \ref{le}, and Lemma \ref{lemma:step3_conv_rates} complete the proof of the main result:
\begin{theorem}
	We have
	\begin{subequations}
		\begin{align}
		\|\bm q-\bm q_h\|_{\mathcal T_h}&\lesssim h^{k+1}(|\bm q|_{k+1}+|y|_{k+1}+|\bm p|_{k+1}+|z|_{k+1}),\label{err_q}\\
		\|\bm p-\bm p_h\|_{\mathcal T_h}&\lesssim h^{k+1}(|\bm q|_{k+1}+|y|_{k+1}+|\bm p|_{k+1}+|z|_{k+1})\label{err_p}.
		\end{align}
	\end{subequations}
\end{theorem}

\section{Numerical Experiments}
\label{sec:numerics}

In this section, we present three numerical examples to confirm our theoretical results. We consider two 2D problems on a square domain $\Omega = [0,1]\times [0,1] \subset \mathbb{R}^2$, and a 3D problem on a cubic domain $\Omega = [0,1]\times [0,1]\times [0,1] \subset \mathbb{R}^3$.  For the three examples, we take $\gamma = 1$ and specify the exact state, dual state, and function $\bm \beta$.  The data $f$, $ g $, and $y_d$ is generated from the optimality system \eqref{eq_adeq}.  Also, we chose $ \tau_1 = 1 $ and set $ \tau_2 $ using \textbf{(A2)}.  For all three examples, conditions \textbf{(A1)}-\textbf{(A3)} are satisfied.

Numerical results for $ k = 0 $ and $ k = 1 $ for the three examples are shown in Table \ref{table_1}--Table \ref{table_6}.  The observed convergence rates exactly match the theoretical results.

\begin{example}\label{example1}
	We take $\bm{\beta} = [1,1]$, state $ y(x_1,x_2) = \sin(\pi x_1) $, and dual state $ z(x_1,x_2) = \sin(\pi x_1)\sin(\pi x_2)$. 
	\begin{table}
		\begin{center}
			\begin{tabular}{|c|c|c|c|c|c|}
				\hline
				$h/\sqrt 2$ &$1/8$& $1/16$&$1/32$ &$1/64$ & $1/128$ \\
				\hline
				$\norm{\bm{q}-\bm{q}_h}_{0,\Omega}$&1.7818e-01   &8.6412e-02   &4.2357e-02   &2.0948e-02   &1.0415e-02 \\
				\hline
				order&-& 1.04& 1.03  &1.02& 1.00\\
				\hline
				$\norm{\bm{p}-\bm{p}_h}_{0,\Omega}$& 4.2057e-01   &2.1839e-01   &1.1116e-01   &5.6062e-02   &2.8151e-02 \\
				\hline
				order&-&  0.94&0.97 &0.99 & 1.00 \\
				\hline
				$\norm{{y}-{y}_h}_{0,\Omega}$&1.6300e-01   &8.4087e-02   &4.2612e-02   &2.1437e-02   &1.0750e-02\\
				\hline
				order&-& 0.95&0.98&0.99 & 1.00 \\
				\hline
				$\norm{{z}-{z}_h}_{0,\Omega}$& 2.1310e-01   &1.0803e-01   &5.4219e-02   &2.7138e-02   &1.3573e-02 \\
				\hline
				order&-& 0.98&0.99&1.00& 1.00 \\
				\hline
			\end{tabular}
		\end{center}
		\caption{Example \ref{example1}: Errors for the state $y$, adjoint state $z$, and the fluxes $\bm q$ and $\bm p$ when $k=0$.}\label{table_1}
	\end{table}

	\begin{table}
		\begin{center}
			\begin{tabular}{|c|c|c|c|c|c|}
				\hline
				$h/\sqrt 2$ &$1/8$& $1/16$&$1/32$ &$1/64$ & $1/128$ \\
				\hline
				$\norm{\bm{q}-\bm{q}_h}_{0,\Omega}$&1.3708e-02 &3.5192e-03   &8.8851e-04   &2.2301e-04   &5.5850e-05 \\
				\hline
				order&-& 2.00& 2.00  &2.00& 2.00\\
				\hline
				$\norm{\bm{p}-\bm{p}_h}_{0,\Omega}$& 3.4995e-02   &8.9472e-03   &2.2581e-03   &5.6694e-04   &1.4202e-04 \\
				\hline
				order&-&  2.00&2.00 &2.00 & 2.00 \\
				\hline
				$\norm{{y}-{y}_h}_{0,\Omega}$&1.1705e-02   &2.9528e-03   &7.4012e-04   &1.8519e-04   &4.6315e-05\\
				\hline
				order&-& 2.00&2.00&2.00 & 2.00 \\
				\hline
				$\norm{{z}-{z}_h}_{0,\Omega}$& 2.3361e-02   &5.9059e-03   &1.4810e-03   &3.7059e-04   &9.2676e-05 \\
				\hline
				order&-& 2.00&2.00&2.00& 2.00 \\
				\hline
			\end{tabular}
		\end{center}
		\caption{Example \ref{example1}: Errors for the state $y$, adjoint state $z$, and the fluxes $\bm q$ and $\bm p$ when $k=1$.}\label{table_2}
	\end{table}
\end{example}

\begin{example}\label{example2}
	We take $\bm{\beta} = [x_2,x_1]$, state $ y(x_1,x_2) = \sin(\pi x_1) $, and dual state $ z(x_1,x_2) = \sin(\pi x_1)\sin(\pi x_2)$. 
	\begin{table}
		\begin{center}
			\begin{tabular}{|c|c|c|c|c|c|}
				\hline
				$h/\sqrt 2$ &$1/8$& $1/16$&$1/32$ &$1/64$ & $1/128$ \\
				\hline
				$\norm{\bm{q}-\bm{q}_h}_{0,\Omega}$ &1.7838e-01   &8.6461e-02   &4.2375e-02   &2.0957e-02   &1.0419e-02 \\
				\hline
				order&-& 1.04& 1.03  &1.02& 1.00\\
				\hline
				$\norm{\bm{p}-\bm{p}_h}_{0,\Omega}$& 4.2050e-01   &2.1848e-01   &1.1123e-01   &5.6101e-02   &2.8171e-02 \\
				\hline
				order&-&  0.95&0.97 &0.99 & 0.99 \\
				\hline
				$\norm{{y}-{y}_h}_{0,\Omega}$&1.6285e-01   &8.4032e-02   &4.2588e-02   &2.1426e-02   &1.0744e-02\\
				\hline
				order&-& 0.95&0.98&0.99 & 1.00 \\
				\hline
				$\norm{{z}-{z}_h}_{0,\Omega}$& 2.1223e-01   &1.0773e-01   &5.4094e-02   &2.7081e-02   &1.3546e-02 \\
				\hline
				order&-& 0.98&0.99&1.00& 1.00 \\
				\hline
			\end{tabular}
		\end{center}
		\caption{Example \ref{example2}: Errors for the state $y$, adjoint state $z$, and the fluxes $\bm q$ and $\bm p$ when $k=0$.}\label{table_3}
	\end{table}

	\begin{table}
		\begin{center}
			\begin{tabular}{|c|c|c|c|c|c|}
				\hline
				$h/\sqrt 2$ &$1/8$& $1/16$&$1/32$ &$1/64$ & $1/128$ \\
				\hline
				$\norm{\bm{q}-\bm{q}_h}_{0,\Omega}$&1.3713e-02   &3.5195e-03   &8.8853e-04   &2.2301e-04   &5.5850e-05 \\
				\hline
				order&-& 2.00& 2.00  &2.00& 2.00\\
				\hline
				$\norm{\bm{p}-\bm{p}_h}_{0,\Omega}$ & 3.5010e-02   &8.9481e-03   &2.2581e-03   &5.6694e-04   &1.4202e-04 \\
				\hline
				order&-&  2.00&2.00 &2.00 & 2.00 \\
				\hline
				$\norm{{y}-{y}_h}_{0,\Omega}$&1.1712e-02   &2.9532e-03   &7.4015e-04   &1.8520e-04   &4.6315e-05\\
				\hline
				order&-& 2.00&2.00&2.00 & 2.00 \\
				\hline
				$\norm{{z}-{z}_h}_{0,\Omega}$& 2.3368e-02   &5.9064e-03   &1.4810e-03   &3.7059e-04   &9.2676e-05 \\
				\hline
				order&-& 2.00&2.00&2.00& 2.00 \\
				\hline
			\end{tabular}
		\end{center}
		\caption{Example \ref{example2}: Errors for the state $y$, adjoint state $z$, and the fluxes $\bm q$ and $\bm p$ when $k=1$.}\label{table_4}
	\end{table}
\end{example}

\begin{example}\label{example3}
	We take $\bm{\beta} = [1,1,1]$, state $ y(x_1,x_2,x_3) = \sin(\pi x_1) $, and dual state $ z(x_1,x_2,x_3) = \sin(\pi x_1)\sin(\pi x_2)\sin(\pi x_3)$. 
	\begin{table}
		\begin{center}
			\begin{tabular}{|c|c|c|c|c|c|}
				\hline
				$h/\sqrt 2$ &$1/2$& $1/4$&$1/8$ &$1/16$ & $1/32$ \\
				\hline
				$\norm{\bm{q}-\bm{q}_h}_{0,\Omega}$ &6.3167e-01   &3.4472e-01 &1.7715e-01   &8.9373e-02   &4.4778e-02 \\
				\hline
				order&-& 0.87& 0.96 &0.99& 1.00\\
				\hline
				$\norm{\bm{p}-\bm{p}_h}_{0,\Omega}$& 4.9907e-01   &2.9505e-01   &1.5339e-01   &7.7393e-02   &3.8724e-02 \\
				\hline
				order&-&  0.76&0.94 &0.99 & 1.00 \\
				\hline
				$\norm{{y}-{y}_h}_{0,\Omega}$&1.7959e-01   &1.0026e-01   &5.3061e-02   &2.7275e-02   &1.3646e-02\\
				\hline
				order&-& 0.84&0.92&0.96 & 1.00 \\
				\hline
				$\norm{{z}-{z}_h}_{0,\Omega}$& 2.3121e-01   &1.3646e-01   &7.2318e-02   &3.7004e-02   &1.8587e-02 \\
				\hline
				order&-& 0.76&0.92&0.97& 1.00 \\
				\hline
			\end{tabular}
		\end{center}
		\caption{Example \ref{example3}: Errors for the state $y$, adjoint state $z$, and the fluxes $\bm q$ and $\bm p$ when $k=0$.}\label{table_5}
	\end{table}

	\begin{table}
				\begin{center}
			\begin{tabular}{|c|c|c|c|c|c|}
				\hline
				$h/\sqrt 2$ &$1/2$& $1/4$&$1/8$ &$1/16$ & $1/32$ \\
				\hline
				$\norm{\bm{q}-\bm{q}_h}_{0,\Omega}$&9.2498e-02   &2.7594e-02   &7.4959e-03   &1.9486e-03   &4.8720e-04\\
				\hline
				order&-& 1.75& 1.90  &1.94& 2.00\\
				\hline
				$\norm{\bm{p}-\bm{p}_h}_{0,\Omega}$ & 1.8360e-01   &5.3637e-02   &1.3921e-02   &3.5138e-03   &8.7857e-04 \\
				\hline
				order&-&  1.80&1.95 &1.99 & 2.00 \\
				\hline
				$\norm{{y}-{y}_h}_{0,\Omega}$&4.4822e-02   &1.1780e-02   &2.9545e-03   &7.3644e-04   &1.8423e-04\\
				\hline
				order&-& 1.93&2.00&2.00 & 2.00 \\
				\hline
				$\norm{{z}-{z}_h}_{0,\Omega}$& 9.1413e-02   &2.7583e-02   &7.3069e-03   &1.8623e-03   &4.6575e-04 \\
				\hline
				order&-& 1.73&1.92&1.97& 2.00 \\
				\hline
			\end{tabular}
		\end{center}
		\caption{Example \ref{example3}: Errors for the state $y$, adjoint state $z$, and the fluxes $\bm q$ and $\bm p$ when $k=1$.}\label{table_6}
	\end{table}
\end{example}

\section{Conclusions}
We proposed an HDG method to approximate the solution of an optimal distributed control problems for an elliptic convection diffusion equation. We obtained optimal a priori error estimates for the control, state, dual state, and their fluxes. The next step is to study optimal control problems governed by more complicated PDEs governing fluids.  It would also be of interest to investigate if postprocessing gives superconvergence for this optimal control problem.


\section*{Appendix}
\label{sec:local_solver_details}

Before we investigate the local elimination, we give the following proposition.
\begin{proposition}
	The matrices $A_{12}$ and $A_{13}$ in \eqref{system_equation} are positive definite.
\end{proposition}
\begin{proof}
	We only prove $A_{12}$ is positive definite; a similar argument applies to $A_{13}$.  The matrix $A_{12}$ is positive definite if and only if $\bm x^TA_{12}\bm x>0$ for any $\bm x=[x_1,x_2,\cdots,x_{N_2}]\in\mathbb R^{N_2} $. For $ x = \sum_{j=1}^{N_2} x_j \phi_j$, we have
	\begin{align*}
	\bm x^T A_{12}\bm x = \langle \tau_1x, x\rangle_{\partial \mathcal T_h} - (\bm{\beta}x, \nabla x)_{\mathcal T_h}.
	\end{align*}
	Moreover
	\begin{align*}
	(\bm \beta x,\nabla x)_{\mathcal T_h}=\langle\bm \beta\cdot\bm n x,x\rangle_{\partial\mathcal T_h}-(\bm \beta x,\nabla x)_{\mathcal T_h},
	\end{align*}
	this implies
	\begin{align*}
	(\bm \beta x,\nabla x)_{\mathcal T_h}&=\frac12\langle\bm \beta\cdot\bm n x,x\rangle_{\partial\mathcal T_h}.
	\end{align*}
	Then,
	\begin{align*}
	\bm x^T A_{12}\bm x  = \langle (\tau_1-\frac 1 2 \bm{\beta}\cdot\bm n) x, x\rangle_{\partial\mathcal T_h}>0,
	\end{align*}
	by the assumption concerning $ \tau_1 $.
\end{proof}

By simple algebraic operations in equation \eqref{system_equation2}, we obtain the following formulas for the matrices $ G_1 $, $ G_2 $, $ H_1 $, and $ H_2 $ in \eqref{local_solver}:
\begin{align*}
G_1 &= B_1^{-1}B_2(B_4+B_2^TB_1^{-1}B_2)^{-1}(B_5+B_2^TB_1^{-1}B_3)-B_1^{-1}B_3,\\
G_2 &= -(B_4+B_2^TB_1^{-1}B_2)^{-1}(B_5+B_2^TB_1^{-1}B_3),\\
H_1 &= -B_1^{-1}B_2(B_4+B_2^TB_1^{-1}B_2)^{-1},\\
H_2 &= (B_4+B_2^TB_1^{-1}B_2)^{-1}.
\end{align*}
We briefly describe how these matrices can be easily computed using the HDG method described in this work.

Since the spaces $ \bm{V}_h $ and $ W_h $ consist of discontinuous polynomials, some of the system matrices are block diagonal and each block is small and symmetric positive definite.  The matrix $ B_1 $ is this type, and therefore $ B_1^{-1} $ is easily computed and is also a matrix of the same type.  Therefore, the the matrices $ G_1 $, $ G_2 $, $ H_1 $, and $ H_2 $ are easily computed if $ B_4 + B_2^T B_1^{-1} B_2 $ is also easily inverted.

It can be checked that $B_2^T B_1^{-1} B_2$ is block diagonal with small nonnegative definite blocks.  Next, $B_4 = \begin{bmatrix}
A_{12} & -\gamma^{-1}A_4\\
A_4 & A_{13}
\end{bmatrix}$, where $A_4$ is symmetric positive block diagonal, $A_{12}$ and $A_{13}$ are positive block diagonal.  Due to the structure of $ B_1 $ and $ B_2 $, the matrix $B_2^TB_1^{-1}B_2 + B_4$ has the form
$\begin{bmatrix}
C_1 & -\gamma^{-1}A_4\\
A_4 & C_2
\end{bmatrix},
$
where $ C_1 $ and $ C_2 $ are symmetric positive block diagonal.  The inverse can be easily computed using the formula
$$\begin{bmatrix}
C_1 & -\gamma^{-1}A_{4}\\
A_4 & C_2
\end{bmatrix}^{-1} = \\
\begin{bmatrix}
C_1^{-1}-\gamma^{-1}C_1^{-1}A_4 D^{-1}A_4C_1^{-1} & \gamma^{-1}C_1^{-1}A_4 D^{-1}\\
-D^{-1}A_4C_1^{-1}  &D^{-1}
\end{bmatrix},
$$
where $D = C_2 +\gamma^{-1}A_4C_1^{-1}A_4$.
Furthermore, $ C_1^{-1} $ and $D^{-1}$ are both symmetric positive block diagonal.

\bibliographystyle{plain}
\bibliography{yangwen_ref_papers,yangwen_ref_books}
\end{document}